\documentclass{amsart}

\usepackage{amsmath,amsthm,mathrsfs,amssymb,graphicx,comment}
\usepackage[all]{xy}
\usepackage{tikz}
\usetikzlibrary{shapes,snakes}
\usepackage[pdfborder={0 0 0},backref=false,colorlinks=true,linktocpage=true]{hyperref}

\newtheorem{theorem}{Theorem}[section]
\newtheorem{proposition}[theorem]{Proposition}
\newtheorem{lemma}[theorem]{Lemma}
\newtheorem{corollary}[theorem]{Corollary}
\theoremstyle{definition}
\newtheorem{definition}[theorem]{Definition}
\newtheorem{statement}[theorem]{Statement}
\newtheorem{question}[theorem]{Question}

\newtheorem{remark}[theorem]{Remark}

\newcommand{\N}{\mathbb{N}}

\newcommand{\res}{\mathbin{\upharpoonright}}

\mathchardef\mhyphen="2D

\newcommand{\seq}[1]{\langle #1 \rangle}
\newcommand{\set}[1]{\{ #1 \}}
\newcommand{\tuple}[1]{\mathbf{ #1 }}

\newcommand{\0}{\mathbf{0}}

\newcommand{\cyl}[1]{[\hspace{-1.5pt}[ #1 ]\hspace{-1.5pt}]}

\newcommand{\RCA}{\mathsf{RCA}_0}
\newcommand{\ACA}{\mathsf{ACA}_0}
\newcommand{\WKL}{\mathsf{WKL}}
\newcommand{\WWKL}{\mathsf{WWKL}}
\newcommand{\ATR}{\mathsf{ATR}_0}
\newcommand{\CA}{\mathsf{CA}_0}
\newcommand{\D}{\mathsf{D}}

\newcommand{\RT}{\mathsf{RT}}
\newcommand{\SRT}{\mathsf{SRT}}

\newcommand{\COH}{\mathsf{COH}}
\newcommand{\CAC}{\mathsf{CAC}}
\newcommand{\ADS}{\mathsf{ADS}}

\newcommand{\TS}{\mathsf{TS}}
\newcommand{\CHO}{\mathsf{SHER}}
\newcommand{\PS}{\mathsf{STRIV}}

\newcommand{\RRT}{\mathsf{RRT}}

\newcommand{\Seq}{\mathsf{Seq}}

\newcommand{\red}{\leq_{\mathrm{W}}}
\newcommand{\sred}{\leq_{\mathrm{sW}}}
\newcommand{\nred}{\nleq_{\mathrm{W}}}
\newcommand{\nsred}{\nleq_{\mathrm{sW}}}
\newcommand{\after}{\bullet}

\begin{document}

\title{On uniform relationships between combinatorial problems}

\author[Dorais]{Fran\c{c}ois G. Dorais}
\address{Department of Mathematics\\
Dartmouth College\\
Hanover, New Hampshire\\
U.S.A.}
\email{francois.g.dorais@dartmouth.edu}

\author[Dzhafarov]{Damir D. Dzhafarov}
\address{Department of Mathematics\\
University of Connecticut\\
Storrs, Connecticut\\
U.S.A.}
\email{damir@math.uconn.edu}

\author[Hirst]{Jeffry L. Hirst}
\address{Department of Mathematical Sciences\\
Appalachian State University\\
Boone, North Carolina\\
U.S.A.}
\email{jlh@math.appstate.edu}

\author[Mileti]{\\ Joseph R. Mileti}
\address{Department of Mathematics and Statistics\\
Grinnell College\\
Grinnell, Iowa\\
U.S.A.}
\email{miletijo@grinnell.edu}

\author[Shafer]{Paul Shafer}
\address{Department of Mathematics\\
Ghent University\\
Ghent\\
Belgium}
\email{paul.shafer@ugent.be}

\thanks{Dzhafarov was partially supported by an NSF Postdoctoral Fellowship. Hirst was partially supported by grant ID\#20800 from the John Templeton Foundation.  (The opinions expressed in this publication are those of the authors and do not necessarily reflect the views of the John Templeton Foundation.) Shafer was supported by the Fondation Sciences Math\'{e}matiques de Paris and is also an FWO Pegasus Long Postdoctoral researcher. We are grateful to D.~Hirschfeldt and C.~Jockusch for numerous helpful comments and discussions during the preparation of this article, and for a remark that helped strengthen Proposition \ref{P:qWWKL}. We thank J.~Miller for pointing out Kummer's theorem, Theorem \ref{T:Kummer}, to us. We also thank V.~Brattka, A.~Montalb\'{a}n, A.~Marcone, C.~Mummert, and the anonymous referee for bringing to our attention the connections of our work with Weihrauch reducibility, of which we were initially unaware. We additionally thank the referee for a number of other useful comments.}

\begin{abstract}
The enterprise of comparing mathematical theorems according to their logical strength is an active area in mathematical logic, with one of the most common frameworks for doing so being reverse mathematics. In this setting, one investigates which theorems provably imply which others in a weak formal theory roughly corresponding to computable mathematics. Since the proofs of such implications take place in classical logic, they may in principle involve appeals to multiple applications of a particular theorem, or to non-uniform decisions about how to proceed in a given construction. In practice, however, if a theorem $\mathsf{Q}$ implies a theorem $\mathsf{P}$, it is usually because there is a direct uniform translation of the problems represented by $\mathsf{P}$ into the problems represented by $\mathsf{Q}$, in a precise sense formalized by Weihrauch reducibility. We study this notion of uniform reducibility in the context of several natural combinatorial problems, and compare and contrast it with the traditional notion of implication in reverse mathematics. We show, for instance, that for all $n,j,k \geq 1$, if $j < k$ then Ramsey's theorem for $n$-tuples and $k$ many colors is not uniformly, or Weihrauch, reducible to Ramsey's theorem for $n$-tuples and $j$ many colors. The two theorems are classically equivalent, so our analysis gives a genuinely finer metric by which to gauge the relative strength of mathematical propositions. We also study Weak K\"{o}nig's Lemma, the Thin Set Theorem, and the Rainbow Ramsey's Theorem, along with a number of their variants investigated in the literature. Weihrauch reducibility turns out to be connected with sequential forms of mathematical principles, where one wishes to solve infinitely many instances of a particular problem simultaneously. We exploit this connection to uncover new points of difference between combinatorial problems previously thought to be more closely related.
\end{abstract}

\maketitle

\section{Introduction}

The idea of reducing, or translating, one mathematical problem to another, with the aim of using solutions to the latter to obtain solutions to the former, is a basic and natural one in all areas of mathematics. For instance, the convolution of two functions can be reduced to a pointwise product via the Fourier transform; the study of a linear operator over a complex vector space can be reduced to the study of a matrix in Jordan normal form, via a change of basis; etc. In general, the precise forms of such reductions vary greatly with the particular problems, but they tend to be most useful when they are constructive or uniform in some appropriate sense. Typically, such reductions preserve various fundamental properties and yield more information, and they are usually easier to implement. These ideas have materialized in many areas such as category theory, complexity theory, proof theory, and set theory (see \cite{Blass-1995}). In this article, we investigate similar uniform reductions between various combinatorial problems in the setting of computability theory, reverse mathematics and computable analysis.

The program of reverse mathematics provides a unified and elegant way to compare the strengths of many mathematical theorems. Its setting is second-order arithmetic, which is a system strong enough to encompass most of classical mathematics. The formalism permits talking about natural numbers and about sets of natural numbers, and hence readily accommodates countable analogues of mathematical propositions. The fundamental idea is to calibrate the proof-theoretical strength of such propositions by classifying which set-existence axioms are needed to establish the structures needed in their proofs. In practice, we work with fragments, or subsystems, of second-order arithmetic, first finding the weakest one that suffices to prove a given theorem, and then obtaining sharpness by showing that the theorem is in fact equivalent to it. Each of the subsystems corresponds to a natural closure point under logical, and more specifically, computability-theoretic, operations. Thus, the base system, Recursive Comprehension Axiom ($\RCA$), roughly corresponds to computable or constructive mathematics; the system Weak K\"{o}nig's Lemma ($\WKL_0$) corresponds to closure under taking infinite paths through infinite binary trees; and the Arithmetical Comprehension Axiom ($\ACA$) corresponds to closure under arithmetical definability, or equivalently, under applications of the Turing jump. Other common subsystems, $\ATR$ and $\Pi^1_1$-$\CA$, which we shall not consider in this article, admit similar characterizations. The point is that there is a rich interaction between proof systems on the one hand, and computability on the other.

We refer the reader to Simpson~\cite{Simpson-2009} for background on reverse mathematics, to Soare~\cite{Soare-1987} for background on computability theory, and to Weihrauch~\cite{Weihrauch-2000} for background in computable analysis. For background on algorithmic randomness, to which some of our results in Sections~\ref{S:WWKL} and~\ref{S:RRT} will pertain, we refer to Downey and Hirschfeldt~\cite{DH-2010}.

In the context of reverse mathematics, we can say that a theorem $\mathsf{P}$ ``reduces'' to a theorem $\mathsf{Q}$ if there is a proof of $\mathsf{P}$ assuming $\mathsf{Q}$ over $\RCA$.  Since these proofs are carried out in a formal system, such a proof of $\mathsf{P}$ from $\mathsf{Q}$ may use $\mathsf{Q}$ several times to obtain $\mathsf{P}$, or may involve non-uniform decisions about which sets to use in a construction.  However, in many natural cases, a proof of $\mathsf{P}$ from $\mathsf{Q}$ uses direct, computable, and uniform translations between problems represented by $\mathsf{P}$ into problems represented by $\mathsf{Q}$.

To describe these types of arguments more precisely, we restrict our focus to $\Pi^1_2$ statements in the language of second-order arithmetic, i.e.,~statements of the form
\[
(\forall X)(\exists Y) \varphi(X,Y),
\]
where $\varphi$ is arithmetical. Each such principle has associated to it a natural class of \emph{instances}, and for each instance, a natural class of \emph{solutions} to that instance.  The following are a few important examples.

\begin{statement}[$\WKL$]
Every infinite subtree of $2^{<\omega}$ has an infinite path.
\end{statement}

\begin{statement}[$\WWKL$]
Every subtree $T$ of $2^{<\omega}$ such that
\[
\frac{|\{\sigma \in 2^n : \sigma \in T\}|}{2^n}
\]
is uniformly bounded away from zero for all $n$ has an infinite path.
\end{statement}

\begin{statement}[Ramsey's Theorem]
Fix $n,k \geq 1$. $\RT^n_k$ is the statement that for every $f \colon [\omega]^n \to k$, there exists an infinite set $H$ (called \emph{homogeneous} for $f$) such that $f$ is constant on $[H]^n$.
\end{statement}

\begin{statement}[$\COH$]
For every sequence of sets $\langle R_i : i \in \omega \rangle$, there exists an infinite set $C$ such that for all $i$, either $C \cap R_i$ is finite or $C \cap \overline{R_i}$ is finite.
\end{statement}

The idea of uniform direct translations alluded to above was made precise by Weihrauch \cite{Weihrauch-1992,Weihrauch-1992a} in the realm of computable analysis and has been widely studied ever since (see \cite[Section 1]{Brattka-Gherardi-2011wd} for a partial bibliography).
In this context, a $\Pi^1_2$ statement $(\forall X)(\exists Y)\varphi(X,Y)$ as above is viewed as a function specification; a partial realizer of such a specification is a function $F$ such that $\varphi(X,F(X))$ holds for all $X$ in the domain of $F$.
For this reason, it is traditional in this context to understand the relation $\varphi(X,Y)$ as a partial multi-valued function where $\varphi(X,Y)$ holds exactly when $Y$ is one of the possible values of the function at $X$.
Weihrauch then introduced a notion of computable reducibility between partial multi-valued functions whereby there are computable processes that serve to uniformly translate realizers of one partial multi-valued function into realizers of another partial multi-valued function. We shall use here the following equivalent definition, which may appear more familiar from perspectives outside of computable analysis, particularly reverse mathematics. However, with a view towards encouraging more collaboration between these two similary-motivated but thus far largely separate approaches, we include an equivalence of the definitions in Appendix \ref{A:equivalence}.


\begin{definition}\label{D:uniform_reductions}
Let $\mathsf{P}$ and $\mathsf{Q}$ be $\Pi^1_2$ statements of second-order arithmetic. We say that
\begin{enumerate}
\item \emph{$\mathsf{P}$ is Weihrauch reducible to $\mathsf{Q}$}, and write $\mathsf{P} \red \mathsf{Q}$, if there exist Turing reductions $\Phi$ and $\Psi$ such that whenever $A$ is an instance of $\mathsf{P}$ then $B = \Phi(A)$ is an instance of $\mathsf{Q}$, and whenever $T$ is a solution to $B$ then $S = \Psi(A \oplus T)$ is a solution to $A$.
\item \emph{$\mathsf{P}$ is strongly Weihrauch reducible to $\mathsf{Q}$}, and write $\mathsf{P} \sred \mathsf{Q}$, if there exist Turing reductions $\Phi$ and $\Psi$ such that whenever $A$ is an instance of $\mathsf{P}$ then $B = \Phi(A)$ is an instance of $\mathsf{Q}$, and whenever $T$ is a solution to $B$ then $S = \Psi(T)$ is a solution to $A$.
\end{enumerate}

\end{definition}



In other words, Weihrauch reducibility differs from strong Weihrauch reducibility only in that the ``backwards'' reduction $\Psi$ takes as oracle not only the solution $T$ to the instance $B = \Phi(A)$ of $\mathsf{Q}$, but also the original instance, $A$, of $\mathsf{P}$. The two notions thus agree on computable instances of problems, but not in general. (See also~\cite[Section 1]{Dzhafarov-COH} for a discussion of the distinction between these approaches in the non-uniform case; and \cite[Section 2.2]{Hirschfeldt-TA} for further discussion of (strong) Weihrauch and related reducibilities in the context of computable combinatorics. We note that the notation in these sources differs from ours, with $\leq_{\rm u}$ and $\leq_{\rm su}$ being used in place of $\red$ and $\sred$, respectively.) For most of our results below (with the notable exception of Theorem \ref{T:Ramsey_non-uniform}) it will not matter which of the two reducibility notions we are working with, so to present the strongest possible results, we shall prove reductions for $\sred$, and non-reductions for $\nred$.

It is straightforward to see each of these reducibilities is reflexive and transitive and thus defines a degree structure on $\Pi^1_2$ statements.


One simple example of a strong Weihrauch\ reduction is that $\RT^n_j \sred \RT^n_k$ whenever $j \leq k$ because given $f \colon [\omega]^n \to j$, we may view $f$ as a function $g \colon [\omega]^n \to k$ (by ignoring the additional colors) and then every set homogeneous for $g$ is homogeneous for $f$.  (Thus, here $\Phi$ and $\Psi$ can both be taken to be the identity reduction.) A slightly more interesting example is that $\RT^m_k \sred \RT^n_k$ whenever $m \leq n$.  To see this, given $f \colon [\omega]^m \to k$, define $g \colon [\omega]^n \to k$ by letting $g(x_1,\dots,x_m,\dots,x_n) = f(x_1,\dots,x_m)$ and notice that $g$ is uniformly obtained from $f$ via a Turing functional, and that every set homogenous for $g$ is homogeneous for $f$.

There are also many examples of such reductions using more complicated Turing functionals.  Friedman, Simpson, and Smith~\cite{FSS-1983} showed that if $\mathsf{P}$ is the statement that every commutative ring with identity has a prime ideal, then $\RCA \vdash \WKL \rightarrow \mathsf{P}$.  Adapting the proof of this result, one can show that it is possible to uniformly computably convert a commutative ring $R$ into an infinite tree $T$ such that every path of $T$ is a prime ideal of $R$, and hence that $\mathsf{P} \sred \WKL$.  For another example, Cholak, Jockusch, and Slaman~\cite[Theorem 12.5]{CJS-2001} exhibit a strong Weihrauch\ reduction of $\COH$ to $\RT^2_2$ via a non-trivial $\Phi$.\footnote{The same is not true if $\RT^2_2$ is replaced by the closely related principle $\mathsf{D}^2_2$, introduced in~\cite[Statement 7.8]{CJS-2001}.  This asserts that if $f : [\omega]^2 \to 2$ is \emph{stable}, i.e.,~if for each $x$ the limit of $f(x,y)$ as $y$ tends to infinity exists, then there is an infinite set consisting either entirely of numbers for which this limit is $0$, or entirely of numbers for which this limit is $1$. A recent result by Chong, Slaman, and Yang~\cite[Theorem 2.7]{CSY-SEP} resolves a longstanding open problem by showing that $\RCA \nvdash \D^2_2 \rightarrow \COH$. However, the model they construct to witness the separation is a non-standard one, and so leaves open the question of whether every $\omega$-model of $\RCA + \D^2_2$ is also a model of $\COH$. A typical reason for this being the case would be if $\COH$ were uniformly reducible to $\D^2_2$. However, by results of Dzhafarov~\cite[Theorem 1.5 and Corollary 1.10]{Dzhafarov-COH} it follows that $\COH \nsred \D^2_2$, and more recently, Lerman, Solomon, and Towsner (unpublished) have shown even that $\COH \nred \D^2_2$.}

Despite the fact that many natural implications in reverse mathematics correspond to Weihrauch reductions (even strong Weihrauch reductions), there are certainly examples where an implication holds in reverse mathematics but no Weihrauch reduction exists.  For example, building on work of Jockusch in~\cite{Jockusch-1972b}, it is known that $\RCA \vdash \RT^n_k \leftrightarrow \mathsf{ACA}$ whenever $n \geq 3$ and $k \geq 2$, and in particular that $\RCA \vdash \RT^3_2 \rightarrow \RT^4_2$.  However, $\RT^4_2 \nred \RT^3_2$ because every computable instance of $\RT^3_2$ has a $\emptyset'''$-computable solution, but there is a computable instance of $\RT^4_2$ with no $\emptyset'''$-computable solution (see~\cite[Theorems 5.1 and 5.6]{Jockusch-1972b}).  The underlying reason why this implication holds in reverse mathematics is that $\emptyset'$ can be coded into a computable instance of $\RT^3_2$, and by relativizing and iterating this result (i.e.,~by using multiple nested applications of $\RT^3_2$), one can obtain the several jumps necessary to compute solutions to instances of $\RT^4_2$.

There are also more subtle instances where no Weihrauch reduction exists despite the fact that degrees of solutions to the problems correspond.  For example, Jockusch~\cite[Theorem 6]{Jockusch-1989} showed that for any $k \geq 2$, the degrees of DNR$_k$ functions (i.e.,~functions $f \colon \omega \to k$ such that $f(e) \neq \Phi_e(e)$ for all $e \in \omega$) are the same as the degrees of DNR$_2$ functions, but there is no Weihrauch reduction witnessing this.  More precisely, he showed that given $k \geq 2$, there is no Turing functional $\Phi$ such that $\Phi(g) \in \text{DNR}_k$ for all $g \in \text{DNR}_{k+1}$. If we let $\mathsf{DNR}_k$ be the $\Pi^1_2$ statement ``for every $X$, there exists a $\text{DNR}_k$ function relative to $X$'', then Jockusch's theorem shows that $\mathsf{DNR}_k \nred \mathsf{DNR}_{k+1}$.

A motivating question for this article is what happens when one varies the number of colors in Ramsey's Theorem.  It is well known that if $n \geq 1$ and $j,k \geq 2$, then $\RCA \vdash \RT^n_j \leftrightarrow \RT^n_k$.  For example, to see that $\RCA \vdash \RT^n_2 \rightarrow \RT^n_3$, we can argue as follows.  Suppose that $f \colon [\omega]^n \to 3$.  Define $g \colon [\omega]^n \to 2$ by letting
\[
g(\tuple{x}) =
\begin{cases}
0 & \textrm{if } f(\tuple{x}) \in \{0,1\}, \\
1 & \textrm{if } f(\tuple{x}) = 2.
\end{cases}
\]
By $\RT^n_2$, we may fix a set $H$ such that $H$ is homogenous for $g$.  Now if $g([H]^2) = \{1\}$, then $H$ is homogeneous for $f$.  Otherwise, the function $f \res [H]^2$ is a $2$-coloring of $[H]^2$, so we may apply $\RT^n_2$ a second time to conclude that there is an infinite $I \subseteq H$ such that $I$ is homogeneous for $f$.  Notice that this proof requires two nested applications of $\RT^n_2$ to obtain a solution to $\RT^n_3$.  However, there are no known degree-theoretic differences between homogeneous sets of computable instances of $\RT^n_2$ and homogeneous sets of computable instances of $\RT^n_3$, so it is unclear whether there is a proof of $\RT^n_3$ using one uniform application of $\RT^n_2$.  We prove below in Theorem~\ref{T:Ramsey_non-uniform} that $\RT^n_k \nsred \RT^n_j$ when $j < k$.

Although the same basic idea of (strong) Weihrauch reducibility is used in the contexts of computable combinatorics, computable analysis, and reverse mathematics, there are important differences beyond terminology that the reader should keep in mind when translating back and forth. 
\begin{itemize}
\item In~\cite{Brattka-Gherardi-2011wd}, (strong) Weihrauch reduction is defined not only for partial multi-valued functions but also for abstract collections of partial functions on Baire space. This more general idea has no equivalent formulation in second-order arithmetic and, moreover, only definable relations make sense in the latter context. Thus, reverse mathematics has a limited view of the (strong) Weihrauch degrees considered in computable analysis. In practice, this limitation only surfaces when considering the general structure of (strong) Weihrauch degrees since these degrees, when considered for their own sake, generally correspond to definable relations. In this paper, we will only consider arithmetically-definable relations, which therefore make sense in all contexts.
\item Computable combinatorics and computable analysis work exclusively with the standard natural numbers whereas reverse mathematics also considers non-standard models. Since the base system $\RCA$ only postulates induction for $\Sigma^0_1$ formulas, issues related to induction often occur in translation and it is not the case that every reduction $\mathsf{P} \sred \mathsf{Q}$ translates into a proof that $\RCA \vdash \mathsf{Q} \rightarrow \mathsf{P}$. For example, a direct analysis of the reduction of Cholak, Jockusch, and Slaman showing that $\COH \sred \RT^2_2$ alluded to above appears to use $\Sigma^0_2$-induction in order to verify that homogenous sets for the transformed coloring are indeed cohesive for the given instance, and some additional work is required to verify that the proof goes through $\RCA$ (see Mileti~\cite[Appendix A]{Mileti-2004}).
\item The typical use of oracles varies in the three contexts. In computable combinatorics, results are usually stated without any use of oracles and issues of relativization are discussed where necessary. In computable analysis, both the (strong) Weihrauch reduction above and its continuous analogue, which permits the use of any oracle, are considered. In reverse mathematics, it is customary to allow any oracle that exists in the model under consideration. So the case of (strong) Weihrauch reduction corresponds to the minimal standard model of $\RCA$ where the only sets are the computable ones, and the continuous analogue corresponds to the case of the full standard model of second-order arithmetic where all sets are present.
\end{itemize}
For these reasons, most of our results will be stated in a way that includes all relevant translations, though our proofs will generally focus only on one point of view, with the others left to the reader.

We use standard notations and conventions from computability theory and reverse mathematics.  We identify subsets of $\omega$ with their characteristic functions, and we identify each $n \in \omega$ with its set of predecessors.  Lower-case letters such as $i,j,k,\ell,m,n,x,y,\dots$ denote elements of $\omega$.  Given a set $A \subseteq \omega$, we let $[A]^n$ denote the set of all subsets of $A$ of size $n$.  We use  $\tuple{x},\tuple{y},\dots$ to denote finite subsets of $\omega$, which we identity with the corresponding tuple listing the elements in increasing order.  We write $\tuple{x} < \tuple{y}$ if $\max \tuple{x} < \min \tuple{y}$. Given a Turing functional $\Phi$, we assume that if $\Phi(A)(x)\downarrow$, then $\Phi(A)(y) \downarrow$ for all $y \leq x$.  We say that a Turing functional $\Phi$ is \emph{total} if $\Phi(A)$ is a total function for every $A \in 2^{\omega}$. Given sets $A$ and $B$, we write $\Phi(A,B)$ in place of $\Phi(A \oplus B)$.

\section{The Squashing Theorem and sequential forms}


We can naturally combine two $\Pi^1_2$ principles $\mathsf{P}$ and $\mathsf{Q}$ into one as follows.  We define the {\em parallel product} $\seq{\mathsf{P},\mathsf{Q}}$ to be the $\Pi^1_2$ principle whose instances are pairs $\seq{A,B}$ such that $A$ is an instance of $\mathsf{P}$ and $B$ is an instance of $\mathsf{Q}$, and the solutions to this instance are pairs $\seq{S,T}$ such that $S$ is a solution to $A$ and $T$ is a solution to $B$. Obviously, this can be generalized to combine any number of $\Pi^1_2$ principles, even an infinite number. In particular, one of our interests will be in cases when $\mathsf{P}$ and $\mathsf{Q}$ are the same principle. For $\alpha \in \omega \cup \{\omega\}$, we let \emph{$\alpha$ applications of $\mathsf{P}$}, or $\mathsf{P}^{\alpha}$, refer to the $\Pi^1_2$ principle whose instances are sequences $\seq{A_i : i < \alpha}$ such that each $A_i$ is an instance of $\mathsf{P}$, and the solutions to this instance are sequences $\seq{S_i : i < \alpha}$ such that each $S_i$ is a solution to $A_i$.
The infinite case $\mathsf{P}^\omega$ is sometimes known as the {\em parallelization} of $\mathsf{P}$ and is also denoted $\widehat{\mathsf{P}}$.

Notice that we trivially have $\COH^2 \sred \COH$ because given two sequences of sets $\seq{R_i : i \in \omega}$ and $\seq{S_i : i \in \omega}$, we can uniformly computably interleave them to form the sequence $\seq{T_i : i \in \omega}$ where $T_{2i} = R_i$ and $T_{2i+1} = S_i$, so that any set cohesive for $\seq{T_i : i \in \omega}$ is cohesive for each of $\seq{R_i : i \in \omega}$ and $\seq{S_i : i \in \omega}$.  In fact, using a pairing function, it is easy to see that $\COH^{\omega} \sred \COH$.  For another example, we have that $\WKL^2 \sred \WKL$ as follows.  Given two infinite trees $\seq{T_0,T_1}$, form a new tree $S$ by letting $\sigma \in S$ if the sequence of even bits from $\sigma$ is an element of $T_0$ and the sequence of odd bits from $\sigma$ is an element of $T_1$.  It is straightforward to check that $S$ is an infinite tree uniformly computably obtained from $\langle T_0,T_1 \rangle$, and that given a path $B$ through $S$, the even bits form a path through $T_0$, and the odd bits form a path through $T_1$.  Moreover, using a pairing function again, we can interleave a sequence $\seq{T_i : i \in \omega}$ of infinite trees together to form one infinite tree such that from any path we can uniformly computably obtain paths through each of the original trees, and hence $\WKL^{\omega} \sred \WKL$. (This fact is also a consequence of Theorem~8.2 of \cite{Brattka-Gherardi-2011wd}, which shows that $\WKL$ is strong Weihrauch equivalent to $\mathsf{LLPO}^\omega$; see also Lemma 5 of Hirst~\cite{Hirst-2007} for a formalized version in reverse mathematics.)

We have the following important example using distinct principles.

\begin{proposition}\label{P:combining_ramsey}
If $n,j,k \geq 1$, then $\langle \RT^n_j, \RT^n_k \rangle \sred \RT^n_{jk}$.
\end{proposition}

\begin{proof}
Given $\langle f,g \rangle$ where $f \colon [\omega]^n \to j$ and $g \colon [\omega]^n \to k$, define $h \colon [\omega]^n \to jk$ by $h(\tuple{x}) = \seq{f(\tuple{x}),g(\tuple{x})}$ for all $\tuple{x} \in [\omega]^n$.  Then $h$ is uniformly computable from $\seq{f,g}$, and any infinite homogeneous set for $h$ is also homogeneous for both $f$ and $g$.
\end{proof}

Given a $\Pi^1_2$ principle $\mathsf{P}$, if $\mathsf{P}^2 \sred \mathsf{P}$, then it is
straightforward to see (by repeatedly applying the reduction procedures) that $\mathsf{P}^n \sred \mathsf{P}$ for each fixed $n \in \omega$.  For example, if $n = 4$ and we are given $\langle A_0,A_1,A_2,A_3 \rangle$ where each $A_i$ is an instance of $\mathsf{P}$, then
\[
\Phi(A_0,\Phi(A_1,\Phi(A_2,A_3)))
\]
is an instance of $\mathsf{P}$ uniformly obtained from $\langle A_0,A_1,A_2,A_3 \rangle$, and from any solution to this instance we can repeatedly apply $\Psi$ to uniformly obtain a sequence $\langle S_0,S_1,S_2,S_3 \rangle$ such that each $S_i$ is a solution to $A_i$. (The same is true if $\sred$ is replaced by $\red$.) It is not at all clear, however, whether this process can be continued into the infinite, i.e.,~does $\mathsf{P}^2 \sred \mathsf{P}$ necessarily imply that $\mathsf{P}^{\omega} \sred \mathsf{P}$?  Given a sequence $\seq{A_i : i \in \omega}$ where each $A_i$ is an instance of $\mathsf{P}$, the natural idea is to consider
\[
\Phi(A_0,\Phi(A_1,\Phi(A_2,\Phi(A_3,\dots)))).
\]
Of course, this process clearly fails to converge and so does not actually define an instance of $\mathsf{P}$.  In fact, we will see later that $\mathsf{P}^2 \sred \mathsf{P}$ does not always imply that $\mathsf{P}^{\omega} \sred \mathsf{P}$.

However, if $\mathsf{P}^2 \sred \mathsf{P}$ and $\mathsf{P}$ is reasonably well-behaved, we will prove that such a ``squashing'' of infinitely many applications of $\mathsf{P}$ into one application of $\mathsf{P}$ is indeed possible.  For example, consider $\mathsf{P} = \RT^2_2$.  The idea is to force some convergence in the above computation by approximating the second coordinate of $\Phi$ as follows.  When attempting to simulate $\Phi(A_0,\Phi(A_1,\Phi(A_2,\dots)))$, we approximate the unknown result of $\Phi(A_1,\Phi(A_2,\dots))$ by guessing that it starts as the all zero coloring.  By assuming this and hence that the second argument looks like a string of zeros, we eventually force convergence of $\Phi(A_0,0^n)$ on $0$, at the cost of introducing some finite initial error in the true ``computation'' of $\Phi(A_1,\Phi(A_2,\dots))$.  Since removing finitely many elements from an infinite homogenous set results in an infinite homogeneous set, these finitely many errors we have introduced into the coloring will not be a problem.

More precisely, we will define a sequence $\seq{B_i : i \in \omega}$ of instances of $\mathsf{P}$ (where intuitively $B_i = \Phi(A_i,\Phi(A_{i+1},\cdots))$ beyond some finite error introduced to force convergence), along with a uniformly computable sequence of numbers $\seq{m_i : i \in \omega}$, such that
\[
B_i(x) = \Phi(A_i,B_{i+1})(x) \text{ for all } x \geq m_i.
\]
Now since we no longer have $B_i = \Phi(A_i,B_{i+1})$ (due to the finite error), the Turing functional $\Psi$ may not convert a solution of $B_i$ into a pair of solutions to $A_i$ and $B_{i+1}$.  In order to deal effectively with these finite errors, to ensure that our $B_i$ are actually instances of $\mathsf{P}$, and to ensure that sequence $\seq{m_i : i \in \omega}$ is uniformly computable (and hence can be used as markers for cut-off points), we need to make some assumptions about $\mathsf{P}$.

\begin{definition}\label{D:total_finite_tolerance}
Let $\mathsf{P}$ be a $\Pi^1_2$ principle (or, more generally, any multi-valued function with domain $2^{\omega}$).
\begin{enumerate}
\item\label{l:total} $\mathsf{P}$ is {\em total} if every element of $2^{\omega}$ is (or codes) an instance of $\mathsf{P}$.
\item\label{I:finite_tolerance} $\mathsf{P}$ has {\em finite tolerance} if there exists a Turing functional $\Theta$ such that whenever $B_1$ and $B_2$ are instances of $\mathsf{P}$ with $B_1(x) = B_2(x)$ for all $x \geq m$, and $S_1$ is a solution to $B_1$, then $\Theta(S_1,m)$ is a solution to $B_2$.
\end{enumerate}
\end{definition}

\begin{proposition}\label{P:ramsey_total_tolerant}
For each $n,k \geq 1$, the principle $\RT^n_k$ is total and has finite tolerance.
\end{proposition}

\begin{proof}
We can view every element of $2^{\omega}$ as a valid $k$-coloring through simple coding.  Define $\Theta$ as follows.  Given $m \in \omega$, compute the largest element $\ell$ of any tuple of $[\omega]^n$ coded by a natural number less than $m$, and let $\Theta(S,m) = \{a \in S : a > \ell\}$.  Now if $B_1$ and $B_2$ are colorings of $[\omega]^n$ using $k$ colors such that $B_1(x) = B_2(x)$ for all $x \geq m$, and $S_1$ is an infinite set homogeneous for $B_1$, then $\Theta(S_1,m)$ is also an infinite set and it is homogeneous for $B_2$.
\end{proof}

Another simple example of a total principle with finite tolerance is $\COH$, where in fact we may take $\Theta(S,m) = S$ (because anything cohesive for a given family of sets is also cohesive for any finite modification of that family).

Although we are certainly interested in the case where $\mathsf{P}^2 \sred \mathsf{P}$, i.e.,~when $\langle \mathsf{P},\mathsf{P} \rangle \sred \mathsf{P}$, we will need a slightly more general formulation below.  As above, when $\langle \mathsf{Q},\mathsf{P} \rangle \sred \mathsf{P}$, it is straightforward to see that $\langle \mathsf{Q}^n,\mathsf{P} \rangle \sred \mathsf{P}$ for each fixed $n \in \omega$.  When passing to the infinite case, however, our ``squashing" never reaches the initial instance of $\mathsf{P}$, but in good cases we can conclude that $\mathsf{Q}^{\omega} \sred \mathsf{P}$.  Notice that if $\mathsf{Q} = \mathsf{P}$, this reduces to the case discussed above.

\begin{remark}
As a rule, all results in this section about $\Pi^1_2$ principles could be formulated more generally for any multi-valued function with domain $2^{\omega}$, as in Definition \ref{D:total_finite_tolerance}. For brevity, we shall omit repeatedly stating this.
\end{remark}

\begin{theorem}[Squashing Theorem]\label{T:squashing_theorem}
Let $\mathsf{P}$ and $\mathsf{Q}$ be $\Pi^1_2$ statements, and assume that both are total and that $\mathsf{P}$ has finite tolerance.
\begin{enumerate}
\item If $\langle \mathsf{Q},\mathsf{P} \rangle \sred \mathsf{P}$ then $\mathsf{Q}^{\omega} \sred \mathsf{P}$.
\item If $\langle \mathsf{Q},\mathsf{P} \rangle \red \mathsf{P}$ then $\mathsf{Q}^{\omega} \red \mathsf{P}$.
\end{enumerate}
\end{theorem}

\begin{proof}
We prove~(1), the proof of~(2) being virtually the same (in fact, the argument can be made somewhat simpler because the oracle has access to the original problem). Throughout, if $\sigma,\tau \in 2^{<\omega}$, we write $\sigma\tau$ for the concatenation of $\sigma$ by $\tau$, and $\sigma^\frown\tau$ for the continuation of $\sigma$ by $\tau$, meaning
\[
\sigma^\frown\tau(i) =
\begin{cases}
\sigma(i) & \textrm{if } i < |\sigma|,\\
\tau(i) & \textrm{if } |\sigma| \leq i < |\tau|,
\end{cases}
\]
for all $i < \max\{|\sigma|,|\tau|\}$. For $A \in 2^\omega$, we similarly define $\sigma^\frown A$.

Fix functionals $\Phi$ and $\Psi$ witnessing the fact that $\seq{\mathsf{Q},\mathsf{P}} \sred \mathsf{P}$.  Since $\mathsf{P}$ is total, we may fix a computable instance $C$ of $\mathsf{P}$ (one could take $C$ to be the sequence of all $0$s, but for some particular problems it might be more convenient or natural to use a different $C$).  Given a sequence $\seq{A_i : i \in \omega}$ of instances of $\mathsf{Q}$, we uniformly define a sequence $\seq{B_i : i \in \omega}$ of instances of $\mathsf{P}$ together with a uniformly computable sequence $\seq{m_i : i \in \omega}$ of numbers so that 
\[
B_i = (C \res m_i)^\frown\Phi(A_i,B_{i+1})
\]
for all $i$.  In other words, we will have $B_i(x) = C(x)$ for all $x < m_i$, and $B_i(x) = \Phi(A_i,B_{i+1})(x)$ for all $x \geq m_i$.  We will then use the instance $B_0$ of $\mathsf{P}$ as our transformed version of $\seq{A_i : i \in \omega}$ and show how given a solution $T_0$ of $B_0$, we can uniformly transform $T_0$ into a sequence $\seq{S_i : i \in \omega}$ of solutions to $\seq{A_i : i \in \omega}$.  One subtle but very important point here is that our sequence $\seq{m_i : i \in \omega}$ of cut-off positions will need to be uniformly computable independent of the instances $\seq{A_i : i \in \omega}$, so that we can use them to unravel a solution $T_0$ of $B_0$ without knowledge of the initial instance.

Thus, our first goal is to define the uniformly computable sequence $\seq{m_i : i \in \omega}$.  We proceed in stages, initially letting $m_0 = 0$. At stage $s$, we define $m_{s+1}$.  The goal is to choose $m_{s+1}$ large enough to ensure that all potential $B_i$ for $i \leq s$ will be defined on $s$.  Intuitively, by placing enough of $C$ down in column $s+1$ (i.e.,~at the beginning of $B_{s+1}$), we must eventually see convergence on previous columns through the cascade effect of the nested $\Phi$.  Since we do not have access to the sequence $\seq{A_i : i \in \omega}$, we make essential use of compactness and the fact that $\mathsf{Q}$ is total to handle all potential inputs at once.

To this end, assume $m_t$ has been defined for each $t \leq s$.  First we claim there exists an $n \in \omega$ such that for all $\sigma_0,\ldots,\sigma_s \in 2^n$,
\[
\begin{array}{l}
\Phi(\sigma_s,C \res n)(s) \downarrow,\\
\Phi(\sigma_{s-1},(C \res m_s)^\frown\Phi(\sigma_s,C \res n))(s) \downarrow,\\
\Phi(\sigma_{s-2},(C \res m_{s-1})^\frown\Phi(\sigma_{s-1},(C \res m_s)^\frown\Phi(\sigma_s,C \res n)))(s) \downarrow,
\end{array}
\]
and for general $i \leq s$,
\begin{equation}\label{E:squashing_uniform}
\Phi(\sigma_i, (C \res m_{i+1})^\frown \cdots ^\frown\Phi(\sigma_{s-1},(C \res m_s)^\frown\Phi(\sigma_s,C \res n))\cdots)(s) \downarrow.
\end{equation}
Observe that the set of all such $n$ is closed under successor. Thus, once the claim is proved, we can define $m_{s+1}$ to be the least such $n$ that is greater than $m_t$ for all $t \leq s$ and also greater than $s$ (to ensure that $B_{s+1}$ will be defined on $0,1,\dots,s$ as well). This observation also implies that to prove the claim, it suffices to fix $i \leq s$, and prove that we can effectively find an $n$ such that \eqref{E:squashing_uniform} holds for all $\sigma_i,\ldots,\sigma_s \in 2^n$.

To this end, let $T$ be the set of all tuples $\seq{\sigma_i,\ldots,\sigma_s}$ of binary strings with $|\sigma_i| = \cdots = |\sigma_s|$ such that
\[
\Phi(\sigma_i, (C \res m_{i+1})^\frown \cdots ^\frown\Phi(\sigma_{s-1},(C \res m_s)^\frown\Phi(\sigma_s,C \res |\sigma_i|)\cdots)(s) \uparrow.
\]
Since each of the computations here has a finite string as an oracle, $T$ is a computable set. Furthermore, if $\seq{\tau_i,\ldots,\tau_s}$ is an initial segment of $\seq{\sigma_i,\ldots,\sigma_s}$ under component-wise extension, that is if $\tau_i \preceq \sigma_i,\ldots,\tau_s \preceq \sigma_s$, then $\seq{\tau_i,\ldots,\tau_s}$ belongs to $T$ if $\seq{\sigma_i,\ldots,\sigma_s}$ does. Thus, $T$ is a subtree in $(2^{<\omega})^s$ under component-wise extension.

Now if $T$ is infinite, then it must have an infinite path $\seq{U_i,\ldots,U_s}$, where $U_i,\ldots,U_s \in 2^\omega$ and $\seq{U_i \res k,\ldots,U_s \res k} \in T$ for all $k$. Then by definition of $T$,
\[
\Phi(U_i, (C \res m_{i+1})^\frown \cdots ^\frown\Phi(U_{s-1},(C \res m_s)^\frown\Phi(U_s,C))\cdots)(s) \uparrow.
\]
As $\mathsf{P}$ and $\mathsf{Q}$ are both total, each of $U_i,\ldots,U_s$ are instances of $\mathsf{Q}$, and each of the second components of any $\Phi$ above are instances of $\mathsf{P}$.  In particular,
\[
(C \res m_{i+1})^\frown \cdots ^\frown\Phi(U_{s-1},(C \res m_s)^\frown\Phi(U_s,C))
\]
is an instance $V$ of $\mathsf{P}$, as is $\Phi(U_i,V)$. But then $\Phi(U_i,V)(s)$ cannot be undefined. We conclude that $T$ is finite, whence its height can clearly serve as the desired $n$.  To complete the proof, we note that an index for $T$ as a computable tree can be found uniformly computably from $i$ and $m_0,\ldots,m_s$, and therefore so can $n$.

\medskip
We now define our reduction procedures witnessing that $\mathsf{Q}^{\omega} \sred \mathsf{P}$.  Let $\seq{A_i : i \in \omega}$ be an instance of $\mathsf{Q}^{\omega}$.  From this sequence, we uniformly computably define a sequence $\seq{B_i : i \in \omega}$ of instances of $\mathsf{P}$ as follows.  Again, we proceed by stages, doing nothing at stage $0$.  At stage $s+1$, we define $B_i(s)$ for each $i\leq s$ and define $B_{s+1}$ on $0,1,\dots,s$. If $s < m_i$, we let $B_i(s) = C(s)$. Otherwise, we let
\[
B_i(s) = \Phi(A_i, (C \res m_{i+1})^\frown \cdots ^\frown\Phi(A_{s-1},(C \res m_s)^\frown\Phi(A_s,C \res m_{s+1}))\cdots)(s),
\]
the right-hand of which we know to be convergent by definition of $m_{s+1}$. That is, we have defined
\[
\begin{array}{rll}
B_s(s) & = & \Phi(A_s,C \res m_{s+1})(s),\\
B_{s-1}(s) & = & \Phi(A_{s-1},(C \res m_s)^\frown\Phi(A_s,C \res m_{s+1}))(s),\\
B_{s-2}(s) & = & \Phi(A_{s-2},(C \res m_{s-1})^\frown\Phi(A_{s-1},(C \res m_s)^\frown\Phi(A_s,C \res m_{s+1})))(s),
\end{array}
\]
and so forth. (Each of the $A_t$ in the computations above could also be replaced by $A_t \res m_{s+1}$.)  We also define $B_{s+1}(j) = C(j)$ for all $j \leq s$.  Since, from the next stage on, $B_{s+1}$ will be defined so that $B_{s+1} \res m_{s+1} = C \res m_{s+1}$, it is not difficult to see that we do indeed succeed in arranging $B_i = (C \res m_i)^\frown \Phi(A_i,B_{i+1})$, as desired.  Furthermore, $\seq{B_i : i \in \omega}$ is defined uniformly computably from $\seq{A_i : i \in \omega}$, and each $B_i$ is an instance of $\mathsf{P}$ because $\mathsf{P}$ is total.  In particular, and there is a Turing functional that produces $B_0$ from $\seq{A_i : i \in \omega}$.

\medskip
Let $\Theta$ be a Turing functional witnessing that $\mathsf{P}$ has finite tolerance.  We claim that from any solution to the instance $B_0$ of $\mathsf{P}$, we can uniformly computably obtain a sequence of solutions to $\seq{A_i : i \in \omega}$.  So suppose $T_0$ is any such solution to $B_0$. The idea is to repeatedly apply the reduction $\Theta$ to deal with the finite errors, followed up by $\Psi$ to convert individual solutions to pairs of solutions.  Indeed, since $B_0(x) = \Phi(A_0,B_1)(x)$ for all $x \geq m_0$, we have that $\Theta(T_0,m_0)$ is a solution to $\Phi(A_0,B_1)$.  Thus, $\Psi(\Theta(T_0,m_0)) = \seq{S_0,T_1}$ is such that $S_0$ is a solution to $A_0$, and $T_1$ is a solution to $B_1$.  The first of these, $S_0$, can serve as the first member of our sequence of solutions.  Since $B_1(x) = \Phi(A_1,B_2)(x)$ for all $x \geq m_1$, we have that $\Theta(T_1,m_1)$ is a solution to $\Phi(A_1,B_2)$.  Thus, $\Psi(\Theta(T_1,m_1)) = \seq{S_1,T_2}$ is such that $S_1$ is a solution to $A_1$, and $T_2$ is a solution to $B_2$.  Continuing in this way, we build an entire sequence $\seq{S_i : i \in \omega}$ of solutions to $\seq{A_i : i \in \omega}$, and since $\seq{m_i : i \in \omega}$ is uniformly computable, we do this uniformly computably from $T_0$ alone. The proof is complete.
\end{proof}

The utility of the Squashing Theorem for our purposes, as we shall see in subsequent sections, is that in many cases it allows us to deduce that multiple applications of a given principle cannot be uniformly reduced to one. This is because there is no (strong) Weihrauch reduction of $\omega$ instances of that principle to one, and in general, showing this tends to be easier.

\begin{corollary}\label{C:squashing}
Let $\mathsf{P}$ be a $\Pi^1_2$ principle that is total and has finite tolerance.
\begin{enumerate}
\item If $\mathsf{P}^2 \sred \mathsf{P}$, then $\mathsf{P}^{\omega} \sred \mathsf{P}$.
\item If $\mathsf{P}^2 \red \mathsf{P}$, then $\mathsf{P}^{\omega} \red \mathsf{P}$.
\end{enumerate}
\end{corollary}

\begin{proof}
Apply Theorem~\ref{T:squashing_theorem} with $\mathsf{Q} = \mathsf{P}$.
\end{proof}

\begin{lemma}\label{L:combining_preserves_properties}
Let $\mathsf{P}$ and $\mathsf{Q}$ be $\Pi^1_2$ principles.
\begin{itemize}
\item If both $\mathsf{P}$ and $\mathsf{Q}$ are total, then $\seq{\mathsf{P},\mathsf{Q}}$ is total.
\item If both $\mathsf{P}$ and $\mathsf{Q}$ have finite tolerance, then $\seq{\mathsf{P},\mathsf{Q}}$ has finite tolerance.
\end{itemize}
\end{lemma}

\begin{proof}
Immediate.
\end{proof}

\begin{corollary}\label{C:squashing_general}
Let $\mathsf{P}$ and $\mathsf{Q}$ be $\Pi^1_2$ statements, assume that both are total and that $\mathsf{P}$ has finite tolerance, and let $m \geq 1$ be given.
\begin{enumerate}
\item If $\langle \mathsf{Q},\mathsf{P}^m \rangle \sred \mathsf{P}^m$ then $\mathsf{Q}^{\omega} \sred \mathsf{P}^m$.
\item If $\langle \mathsf{Q},\mathsf{P}^m \rangle \red \mathsf{P}^m$ then $\mathsf{Q}^{\omega} \red \mathsf{P}^m$.
\end{enumerate}
\end{corollary}

\begin{proof}
Repeatedly applying Lemma~\ref{L:combining_preserves_properties}, we see that $\mathsf{P}^m$ is total and has finite tolerance.  The result follows from the Squashing Theorem.
\end{proof}

\begin{corollary}\label{C:squashing_multiple}
Let $\mathsf{P}$ be a $\Pi^1_2$ principle that is total and has finite tolerance, and let $m \geq 1$ be given.
\begin{enumerate}
\item If $\mathsf{P}^{m+1} \sred \mathsf{P}^m$, then $\mathsf{P}^{\omega} \sred \mathsf{P}^m$.
\item If $\mathsf{P}^{m+1} \red \mathsf{P}^m$, then $\mathsf{P}^{\omega} \red \mathsf{P}^m$.
\end{enumerate}
\end{corollary}

\begin{proof}
Since $\mathsf{P}^{m+1} \sred \mathsf{P}^m$, we know that $\langle \mathsf{P},\mathsf{P}^m \rangle \sred \mathsf{P}^m$, so the result follows from the previous corollary.
\end{proof}

For the remainder of this article, we employ the following short-hand to avoid excessive exponents and to give $\mathsf{P}^{\omega}$ a more evocative name.

\begin{statement}\label{S:sequential_form}
For any $\Pi^1_2$ principle $\mathsf{P}$, we denote $\omega$ applications of $\mathsf{P}$, or $\mathsf{P}^{\omega}$, by $\Seq\mathsf{P}$.  We call $\Seq\mathsf{P}$ the \emph{sequential version} of $\mathsf{P}$.
\end{statement}

\noindent So, for instance, Corollary~\ref{C:squashing} says that that if $\mathsf{P}$ is total and has finite tolerance, then $\mathsf{P}^2 \sred \mathsf{P}$ implies that $\Seq\mathsf{P} \sred \mathsf{P}$.  With this terminology, we have the following simple result.

\begin{proposition}\label{P:unif_imlies_seq_unif}
Let $\mathsf{P}$ and $\mathsf{Q}$ be $\Pi^1_2$ principles.
\begin{enumerate}
\item If $\mathsf{P} \sred \mathsf{Q}$, then $\Seq\mathsf{P} \sred \Seq\mathsf{Q}$.
\item If $\mathsf{P} \red \mathsf{Q}$, then $\Seq\mathsf{P} \red \Seq\mathsf{Q}$.
\end{enumerate}
\end{proposition}

\begin{proof}
For~(1), fix $\Phi$ and $\Psi$ witnessing the reduction $\mathsf{P} \sred \mathsf{Q}$.  Given an instance $\seq{A_i : i \in \omega}$ of $\Seq\mathsf{P}$, we have that $\seq{\Phi(A_i) : i \in \omega}$ is an instance of $\Seq\mathsf{Q}$ uniformly computably obtained from it.  Also, if $\seq{T_i : i \in \omega}$ is a solution to $\seq{\Phi(A_i) : i \in \omega}$, then $\seq{\Psi(T_i) : i \in \omega}$ is a solution to $\seq{A_i : i \in \omega}$. For~(2), the proof is the same, except we must take $\seq{\Psi(A_i, T_i) : i \in \omega}$ as the solution.
\end{proof}

\section{Ramsey's theorem for different numbers of colors}

Throughout this section, let $n \geq 1$ be fixed. Our goal is to work up towards a proof of the following theorem.

\begin{theorem}\label{T:Ramsey_non-uniform}
For all $j,k \geq 2$ with $j < k$, we have $\RT^n_k \nsred \RT^n_j$.
\end{theorem}

\noindent As pointed out above, we have that $\RCA \vdash \RT^n_j \rightarrow \RT^n_k$, but the obvious proof uses multiple nested applications of $\RT^n_j$.  Theorem~\ref{T:Ramsey_non-uniform} says that it is impossible to give a uniform proof of this implication using just one application of $\RT^n_j$.

The key ingredients of the proof are Proposition~\ref{P:combining_ramsey}, the Squashing Theorem, and the fact that it is possible to code more into $\Seq\RT^n_k$ than into $\RT^n_k$ alone.  To illustrate the last of these, consider $\RT^1_2$.  Notice that every computable instance of $\RT^1_2$ trivially has a computable solution because either there are infinitely many $0$s or there are are infinitely many $1$s (and each of these sets is computable), but there is one non-uniform bit of information used to determine which of these two statements is true.  However, it is a straightforward matter to build a computable instance of $\Seq\RT^1_2$ such that every solution computes $\emptyset'$.  The idea is to use each column to code one bit of $\emptyset'$ by exploiting this one non-uniform decision.  In fact, for higher exponents this result can be made sharper, as we now prove. (See also \cite[Proposition 47]{KK-2012} for a related result in the context of proof mining and program extraction.)

\begin{lemma}\label{L:jump_coding}
There is a computable instance of $\Seq\RT^n_2$ every solution to which computes $\emptyset^{(n)}$.
\end{lemma}

\begin{proof}
We prove the result for $n$ being odd; the case where $n$ is even is analogous. Fix a computable predicate $\varphi$ such that
\[
\emptyset^{(n)} = \{i \in \omega : (\exists x_0)(\forall x_1) \cdots (\exists x_{n-1})
~\varphi(i,x_0,x_1,\dots,x_{n-1})\}.
\]
Define a computable sequence of colorings $\seq{ f_i : {i \in \omega}}$ by
\begin{align*}
f_i(\tuple{y}) =
\begin{cases}
1 & \text{if $(\exists x_0 < y_0)(\forall x_1 < y_1) \cdots (\exists x_{n-1} < y_{n-1}) ~\varphi(i,x_0,x_1,\dots,x_{n-1})$},\\
0 & \text{otherwise},
\end{cases}
\end{align*}
for all $\tuple{y} = \seq{y_0,y_1,\ldots,y_{n-1}} \in [\omega]^n$.

Let $\seq{ H_i :{i \in \omega}}$ be any sequence of infinite homogeneous sets for the $f_i$. We claim that $\emptyset^{(n)}(i) = f_i([H_i]^n)$ for all $i$, and hence that $\emptyset^{(n)} \leq_T \seq{H_i : i \in \omega}$. To see this, suppose first that $i \in \emptyset^{(n)}$. Let $\seq{w_{2j} :{2j < n}}$ be Skolem functions for membership in $\emptyset^{(n)}$,
so that
\[
(\forall x_1)(\forall x_3)\cdots(\forall x_{n-2})~\varphi(i,w_0(i),x_1,w_2(i,x_1),x_3,\dots,w_{n-1}(i,x_1,x_3,\dots,x_{n-2})).
\]
Now define an increasing sequence $z_0 < z_1 < \cdots < z_{n-1}$ of elements $H_i$ as follows. Start by letting $z_0$ be the least $z \in H_i$ that is greater than $w_0(i)$. Then, given $j$ with $1 \leq j \leq n-1$, suppose we have defined $z_k$ for all $k < j$. If $j$ is odd, let $z_j$ be the least $z \in H_i$ that is greater than $z_{j-1}$. If $j$ is even, let $z_j$ be the least $z \in H_i$ that is greater than $z_{j-1}$, and also greater than $w_j(i,x_1,x_3,\ldots,x_{j-1})$ for all sequences $x_1,x_3,\ldots,x_{j-1}$ with $x_k < z_k$ for each odd $k < j$.

The sequence of $z_j$ so constructed now clearly satisfies
\begin{equation}\label{E:bound}
(\exists x_0 < z_0)(\forall x_1 < z_1) \cdots (\exists x_{n-1} < z_{n-1})~\varphi(i,x_0,x_1,\dots,x_{n-1}).
\end{equation}
So by definition of $f_i$, we have that $f_i(z_0,\ldots,z_{n-1}) = 1$. And since the $z_j$ all belong to $H_i$, it follows that $f([H_i]^n) = 1$, as desired.

Now suppose that $i \notin \emptyset^{(n)}$. We can similarly construct a sequence ${z_0 < \cdots < z_{n-1}}$ of elements of $H_i$ witnessing that $f([H_i]^n) = 0$. Let $\seq{w_{2j+1} : 2j+1 < n}$ be Skolem functions for non-membership in $\emptyset^{(n)}$, so that
\[
(\forall x_0)(\forall x_2)\cdots(\forall x_{n-1})~\neg \varphi(i,x_0,w_1(i,x_0),x_2,\ldots,w_{n-2}(i,x_0,x_2,\dots,x_{n-3}),x_{n-1}).
\]
Let $z_0$ be the least element of $H_i$, and suppose we are given a $j$ with $1 \leq j \leq n-1$ such that $z_k$ has been defined for all $k < j$. If $j$ is even, let $z_j$ be the least $z \in H_i$ that is greater than $z_{j-1}$. If $j$ is odd, let $z_j$ be the least $z \in H_i$ that is greater than $z_{j-1}$, and also greater than $w_j(i,x_0,x_2,\ldots,x_{j-1})$ for all sequences $x_0,x_2,\ldots,x_{j-1}$ with $x_k < z_k$ for each even $k < j$.

This sequence of $z_j$ satisfies the negation of \eqref{E:bound} above, so $f_i(z_0,\ldots,z_{n-1}) = 0$ by definition. Since all the $z_j$ belong to $H_i$, the claim follows.
\end{proof}

After relativization and translation into the language of strong Weihrauch\ reductions, we obtain from the above that $\mathsf{TJ}^n \sred \Seq\RT^n_2.$
(See the discussion following Corollary~\ref{C:TJ} for a definition of the iterated Turing jump, $\mathsf{TJ}^n$.)

\begin{lemma}\label{L:cant_add_coloring}
For all $n \geq 1$ and $k \geq 2$, we have $\langle \RT^n_2,\RT^n_k \rangle \nred \RT^n_k$.
\end{lemma}

\begin{proof}
Suppose instead that $\langle \RT^n_2,\RT^n_k \rangle \red \RT^n_k$.  Since $\RT^n_2$ and $\RT^n_k$ are both total and have finite tolerance by~\ref{P:ramsey_total_tolerant}, we may use the Squashing Theorem~\ref{T:squashing_theorem} to conclude that $\Seq\RT^n_2 \red \RT^n_k$.  Fix $\Phi$ and $\Psi$ witnessing the reduction, and let $f = \seq{f_i : i \in \omega}$ be any computable instance of $\Seq\RT^n_2$.  Apply $\Phi$ to this sequence to obtain an instance $g$ of $\RT^n_k$ and notice that $g$ is computable.  By Theorem 5.6 of Jockusch~\cite{Jockusch-1972b}, we can find an infinite set $H$ homogeneous for $g$ such that $H^\prime \leq_T \emptyset^{(n)}$ (since $\RT^1_k$ is computably true, Jockusch's Theorem 5.6 holds also when $n=1$).  We then have that $S = \seq{S_i : i \in \omega} = \Psi (f,H)$ is a solution to $f = \seq{f_i : i \in \omega}$ with $S' \leq_T \emptyset^{(n)}$.

But as the sequence $f = \seq{f_i : i \in \omega}$ was chosen as an arbitrary computable instance of $\Seq\RT^n_2$, this would imply that every computable instance of $\Seq\RT^n_2$ has a solution with jump computable in $\emptyset^{(n)}$. This contradicts Lemma~\ref{L:jump_coding}, since no such set can compute $\emptyset^{(n)}$.  Therefore, we must have $\langle \RT^n_2,\RT^n_k \rangle \nred \RT^n_k$.
\end{proof}

We shall prove Theorem~\ref{T:Ramsey_non-uniform} by means of the following weaker version of the theorem, which now follows easily.

\begin{corollary}\label{C:double_exponent}
For all $n \geq 1$ and $k \geq 2$, we have $\RT^n_{2k} \nred \RT^n_k$.
\end{corollary}

\begin{proof}
Suppose instead that $\RT^n_{2k} \red \RT^n_k$.  We know from Proposition~\ref{P:combining_ramsey} that $\langle \RT^n_2,\RT^n_k \rangle \red \RT^n_{2k}$.  Hence, using transitivity of $\red$, we have $\langle \RT^n_2,\RT^n_k \rangle \red \RT^n_k$, contrary to Lemma~\ref{L:cant_add_coloring}.
\end{proof}

In order to use this corollary to handle all cases of Theorem~\ref{T:Ramsey_non-uniform}, we use the following result saying that we can fan out a strong Weihrauch\ reduction $\RT^n_k \sred \RT^n_j$ to obtain a strong Weihrauch\ reduction with a larger spread between the number of colors used.

\begin{lemma}\label{L:spread_coloring}
Let $n,j,k,s \geq 1$.  If $\RT^n_k \sred \RT^n_j$, then $\RT^n_{k^s} \sred \RT^n_{j^s}$.
\end{lemma}

\begin{proof}
Fix $\Phi$ and $\Psi$ witnessing the fact that $\RT^n_k \sred \RT^n_j$.  In what follows, define $e(b,a,i)$ for all $b,a \in \omega$ and all $i < \lfloor \log_b a \rfloor$ to be the $i$th digit in the base $b$ expansion of $a$. Thus, for example, $e(10,25,0) = 5$ and $e(2,25,0) = 1$.

Fix an arbitrary $f \colon [\omega]^n \to k^s$.  We now convert $f$ into $s$ many colorings $f_0,\ldots,f_{s-1} \colon [\omega]^n \to k$ by setting
\[
f_i(\tuple{x}) = e(k,f(\tuple{x}),i)
\]
for all $i < s$ and all $\tuple{x} \in [\omega]^n$. Then for any $\tuple{x}$, the expansion of $f(\tuple{x})$ in base $k$ is precisely $f_0(\tuple{x}) \cdots f_{s-1}(\tuple{x})$. Hence, any set that is simultaneously homogeneous for each of the $f_i$ is also homogeneous for $f$.

Now apply the reduction $\Phi$ to each $f_i$ to obtain colorings $g_i \colon [\omega]^n \to j$ for each $i < s$. We merge these $m$ many colorings into one coloring $g \colon [\omega]^n \to j^s$ defined by
\[
g = \sum_{i = 0}^{s-1} j^i g_i.
\]
Notice that any infinite set $H$ homogeneous for $g$ is simultaneously homogeneous for each of the $g_i$. Hence, $\Psi(H)$ is simultaneously homogeneous for each of the $f_i$. But then by the observation above, it follows that $\Psi(H)$ is an infinite homogeneous set for $f$. Since the reduction from $f$ to $g$ was uniformly computable, the lemma is proved.
\end{proof}

We can now prove our main result.

\begin{proof}[Proof of Theorem~\ref{T:Ramsey_non-uniform}]
Seeking a contradiction, fix $j < k$ and assume $\RT^n_k \sred \RT^n_j$.  Since $\frac{k}{j} > 1$, we may fix $s \in \omega$ with $(\frac{k}{j})^s > 4$, so that $4j^s < k^s$.  Let $m \in \omega$ be least such that $j^s \leq 2^m$.  We then have $2^{m-1} < j^s$, so $2^{m+1} < 4j^s < k^s$, and hence
\[
j^s \leq 2^m < 2^{m+1} < k^s.
\]
Since we are assuming $\RT^n_k \sred \RT^n_j$, we can use Lemma~\ref{L:spread_coloring} to conclude that $\RT^n_{k^s} \sred \RT^n_{j^s}$.  We therefore have
\[
\RT^n_{2^{m+1}} \sred \RT^n_{k^s} \sred \RT^n_{j^s} \sred \RT^n_{2^m}
\]
Since $\sred$ is transitive, it follows that $\RT^n_{2^{m+1}} \sred \RT^n_{2^m}$, contradicting Corollary~\ref{C:double_exponent}.
\end{proof}

It is worth pointing out that, in proving of Theorem \ref{T:Ramsey_non-uniform}, the proof of Lemma \ref{L:spread_coloring} was the only moment where it mattered that we were working with the \emph{strong} form of Weihrauch reducibility. Specifically, since $\Psi$ there took solutions to $g_i$ to solutions to $f_i$ for each $i$, in finding a simultaneous solution $H$ for all the $g_i$ we found a simultaneous solution $\Psi(H)$ for all the $f_i$. This would no longer be the case if joining with original instances was permitted, since then we could not guarantee that $\Psi(g_i,H_i)$ would be the same set for all $i$. We do not know how to overcome this difficulty, and hence leave open the question of whether Lemma \ref{L:spread_coloring} and Theorem \ref{T:Ramsey_non-uniform} also holds with $\sred$ replaced by $\red$.

\section{Weak Weak K\"{o}nig's Lemma}\label{S:WWKL}

As discussed in Section 2, it is straightforward to see that $\Seq\WKL \sred \WKL$ (and the reverse direction is obvious).  However, the situation of $\WWKL$ is more interesting.  By performing the same interleaving process to show that $\WKL^2 \sred \WKL$, one checks that the resulting tree has positive measure if each of the two input trees do (in fact, the measure of the interleaved tree is the product of the measures of the original trees), and hence it follows that $\WWKL^2 \sred \WWKL$.  By iterating this, it follows that given any {\em finite} sequence $\seq{T_i : i < n}$ of trees with positive measure many paths, one can interleave them to obtain a tree $S$ whose measure will be the product of the $T_i$ (and hence also positive) such that from any path through $S$, one can uniformly compute paths through the $T_i$.  However, this idea does not carry over to the case of an infinite sequence of trees of positive measure, since then the interleaving process can produce a tree of measure $0$. Indeed, this can happen even if the measures of the trees in the sequences are uniformly bounded away from $0$.

Notice that we trivially have $\WWKL \sred \WKL$, so $\Seq\WWKL \sred \Seq\WKL$ by Proposition~\ref{P:unif_imlies_seq_unif}.  As explained in Section 2, we have $\Seq\WKL \sred \WKL$, and hence $\Seq\WWKL \sred \WKL$ by transitivity of $\sred$.  One can also show that this can be formalized to give $\RCA \vdash \WKL \rightarrow \Seq\WKL \rightarrow \Seq\WWKL$.  The next theorem shows that the converses are also true, and hence $\Seq\WWKL$, even in this weaker form, is in fact strictly stronger than $\WWKL$.

\begin{theorem}\label{T:Equiv_SeqWWKL_WKL}
\
\begin{enumerate}
\item $\WKL \sred \Seq\WWKL$.
\item $\RCA \vdash \Seq\WWKL \rightarrow \WKL$.
\end{enumerate}
In fact, both of these statements hold even if we restrict $\Seq\WWKL$ to infinite sequences of subtrees of $2^{<\omega}$ of measure uniformly bounded away from $0$.
\end{theorem}

\begin{proof}
We prove (2) in the stronger form in order to handle the formalized version carefully, but our construction is completely uniform and hence can be turned into a proof of (1).

Let $S$ be an arbitrary infinite subtree of $2^{<\omega}$.  We define a sequence of trees $\seq{T_{\sigma} : \sigma \in 2^{<\omega}}$ indexed by finite binary strings $\sigma \in 2^{<\omega}$ (which of course can be put in bijection with $\omega$).  Intuitively, $T_{\sigma}$ is constructed as follows.  Put the empty string $\emptyset$, $0$, and $1$ in $T_{\sigma}$.  Keep building above both $0$ and $1$ putting in all possible extensions as long as $\sigma 0$ and $\sigma 1$ both look extendible in $S$.  If we discover that one of $\sigma 0$ or  $\sigma 1$ is not extendible in $S$, then stop building above $0$ or $1$ in $T_{\sigma}$ accordingly, and forever build above the other side (even if the other also ends up not extendible in $S$).  In this way, $T_{\sigma}$ will always have measure either $\frac12$ or $1$.

More formally, we define our sequence as follows.  Given $\rho \in 2^{<\omega}$ and $k \in \omega$, let $Ext_S(\rho,k)$ be the $\Delta_0$ predicate saying that either $k \leq |\rho|$, or there exists an element of $S$ extending $\rho$ of length $k$.  Given $\sigma \in 2^{<\omega}$, define $T_{\sigma}$ to be $\emptyset$ together with the set of $\tau \in 2^{<\omega} \backslash \{\emptyset\}$ satisfying one of the following:
\begin{itemize}
\item $\tau(0) = 0$ and $Ext_S(\sigma 0, |\tau|)$.
\item $\tau(0) = 1$ and $Ext_S(\sigma 1, |\tau|)$.
\item $\tau(0) = 0$ and $(\exists k < |\tau|)[Ext_S(\sigma 0,k) \wedge \neg Ext_S(\sigma 1,k)]$.
\item $\tau(0) = 1$ and $(\exists k < |\tau|)[Ext_S(\sigma 1,k) \wedge \neg Ext_S(\sigma 0,k)]$.
\item $(\exists k < |\tau|)[Ext_S(\sigma 0, k) \wedge Ext_S(\sigma 1, k) \wedge \neg Ext_S(\sigma 0, k+1) \wedge \neg Ext_S(\sigma 1, k+1)]$.
\end{itemize}
Note that the last condition handles the case when both sides die at the same level, and in this situation we (arbitrarily) build the full tree.

Since $S$ is tree, if $k < m$ and $Ext_S(\rho,m)$, then $Ext_S(\rho,k)$.  By $\Sigma_1^0$-induction and the fact that $Ext_S(\rho,0)$ holds by definition, if $\neg Ext_S(\rho,m)$ then there exists a unique $k \in \omega$ with $k < m$ such that $Ext_S(\rho,k)$ and $\neg Ext_S(\rho,k+1)$.  Using these facts, it is straightforward to check that each $T_{\sigma}$ is a tree, and that for each $m \in \omega$, either every element of $2^m$ is in $T_{\sigma}$ or exactly half of the elements of $2^m$ are in $T_{\sigma}$.

Applying $\Seq\WWKL$ to the sequence $\seq{T_{\sigma} : \sigma \in 2^{<\omega}}$, we obtain a sequence $\seq{B_{\sigma} : \sigma \in 2^{<\omega}}$ of paths through the trees $\seq{T_{\sigma} : \sigma \in 2^{<\omega}}$.  We now define a function $C \colon \omega \to \{0,1\}$ recursively by letting $C(n) = B_{C \res n}(0)$, where $C \res n$ is the finite sequence $C(0) C(1) \cdots C(n-1)$.  We claim that $C$ is a path through $S$.  To show this, we prove the stronger fact that for each $n \in \omega$, we have $(\forall m) Ext_S(C \res n, m)$.  The proof is by induction on $n$ (using $\Pi_1^0$-induction, which follows from $\Sigma_1^0$-induction).  For $n = 0$, note that $C \res n = \emptyset$, and we know that $(\forall m) Ext_S(\emptyset,m)$ because $S$ is an infinite tree by assumption.  Suppose that we have a given $n \in \omega$ for which $(\forall m) Ext_S(C \res n, m)$.  In this case, at least one of $(\forall m) Ext_S((C \res n) 0, m)$ or $(\forall m) Ext_S((C \res n) 1, m)$ must hold.  Now if $i \in \{0,1\}$ is such that $\neg Ext_S((C \res n) i, m)$, then $T_{C \res n}$ has no node extending $i$ of length $m$ (by definition of the $T_{\sigma}$), so it must be the case that $B_{C \res n}(0) = 1-i$.  Therefore, we must have $(\forall m) Ext_S(C \res (n+1), m)$.  This completes the induction, and the proof.
\end{proof}

Fact~(1) above can also be derived from the result of Brattka and Gherardi~\cite[Theorem 8.2]{Brattka-Gherardi-2011wd} that $\WKL \sred \Seq\mathsf{LLPO}$ and the observation of Brattka and Pauly~\cite[Figure 1]{Brattka-Pauly-2010} that that $\mathsf{LLPO} \sred \WWKL$. (See \cite[Section 1]{Brattka-Gherardi-2011wd} for a definition of $\mathsf{LLPO}$.)

On the other hand, we have the following fact, which follows in this form from more general results of Brattka and Pauly~\cite[Proposition 22]{Brattka-Pauly-2010}, and also essentially by the proof of Simpson and Yu~\cite{Simpson-Yu-1990} that $\WWKL \nrightarrow \WKL$ over $\RCA$. We include a proof for completeness.

\begin{proposition}
$\WKL \nred \WWKL$.
\end{proposition}

\begin{proof}
By results of Jockusch and Soare~\cite[Theorem 5.3]{JS-1972a}, there is a computable instance of $\WKL$ for which only measure $0$ many elements of $2^{\omega}$ compute a solution.  However, every $1$-random computes an infinite path through every infinite computable instance of $\WWKL$. (See, e.g.,~\cite[Lemma 1.3]{AKLS-2004}.)
\end{proof}



Thus while $\WWKL^n \sred \WWKL$ for each $n \in \omega$, we have that $\Seq\WWKL \nred \WWKL$.  Notice that the Squashing Theorem does not apply to $\WWKL$ because it is not total (there is no clear way to view every real as coding an instance of $\WWKL$).

We now turn to questions about uniformly passing back and forth between trees of positive measure. Consider any such tree $T$ of $2^{<\omega}$. A question that seems natural is whether from a positive rational $q < 1$, it is possible to build a tree $S$ of measure at least $q$, each path through which computes a path through $T$. Intuitively, is it possible to blow up the measure of $T$ without losing information about its paths? It is not difficult to see that the answer is affirmative, and in fact, that such an $S$ can be obtained uniformly from $q$ and an index for $T$. Indeed, fix a universal Martin-L\"{o}f test $\{U_i : i \in \omega\}$ and let $S = 2^{<\omega} - U_i$ for the least $i$ with $q \leq 1-2^{-i}$. Every path through $S$ is 1-random, and hence computes a path through $T$, but not uniformly. The following lemma and proposition show that if we allow $S$ to be defined non-uniformly from $T$ and $q$, then we can arrange for the computations from paths to paths to be uniform.

\begin{lemma}
Given a tree $T \subseteq 2^{<\omega}$ of positive measure $p$, and given $\varepsilon > 0$, there is a tree $S$, each path of which uniformly computes a path through $T$, such that the measure of the complement of $S$ is at most $(1+\varepsilon)(1-p)^2$.
\end{lemma}

\begin{proof}
We may assume $p < 1$, since otherwise we can just take $S = T$. Fix a positive $\delta < 1$ such that $1-\delta p \leq (1+\varepsilon)(1-p)$. Choose minimal, hence incompatible, strings $\sigma_0,\ldots,\sigma_{n-1} \notin T$ such that
\[
\sum_{i<n} 2^{-|\sigma_i|} \geq \delta (1-p),
\]
and let
\[
S = T \cup (\bigcup_{i<n} {\sigma_i}T),
\]
where $\sigma_iT = \{\sigma_i\tau : \tau \in T\}$. Then the measure of the complement of $S$ is
\[
(1-p) - \sum_{i < n} 2^{-|\sigma_i|}p \leq (1-p) - \delta p(1-p) = (1-p)(1-\delta p) \leq (1+\varepsilon)(1-p)^2.
\]
Now let $\Phi$ be the functional that sends $A \in 2^{\omega}$ to $A(|\sigma_i|)A(|\sigma_i|+1)\cdots$ if $\sigma_i \preceq A$ for some $i < n$, and to $A$ otherwise. Clearly, $\Phi(A)$ is a path through $T$ whenever $A$ is a path through $S$.
\end{proof}

\begin{proposition}\label{P:uniform_trees}
Given a tree $T \subseteq 2^{<\omega}$ of positive measure $p$, and given a positive rational $q < 1$, there is a tree $S \subseteq 2^{<\omega}$ of measure at least $q$, each path of which uniformly computes a path through $T$.
\end{proposition}

\begin{proof}
Given $T$, $p$, and $q$, choose $\varepsilon_0,\ldots,\varepsilon_{n-1}$ so that
\begin{equation}\label{E:estimate}
(1+\varepsilon_{n-1}) (1+\varepsilon_{n-2})^2 \cdots (1+\varepsilon_0)^{2(n-1)} (1-p)^{2n} < 1-q.
\end{equation}
Now iterate the lemma. Let $S_{-1} = T$, and given $S_{i-1}$ obtain $S_i$ with complement of measure at most  $(1+\varepsilon_i)(1-\mu(S_{i-1}))^2$ such that each path through $S_i$ computes a path through $S_{i-1}$. By induction, the complement of $S_{n-1}$ has measure bounded by \eqref{E:estimate}, and each path through it computes a path through $T$.
\end{proof}

Thus, we can either uniformly blow up the measure of a given tree $T$, and have paths through the new tree non-uniformly compute paths through the old; or we can non-uniformly blow up the measure of $T$, and have paths through the new tree uniformly compute paths through the old. The following proposition, which is a direct corollary of Theorem~\ref{T:Equiv_SeqWWKL_WKL}, shows that we cannot achieve both types of uniformity simultaneously.

\begin{proposition}\label{P:nonuniform_trees}
There is no effective procedure that, given (an index for) a computable subtree $T$ of $2^{<\omega}$ of positive measure, and a positive rational $q$, produces (an index for) a computable subtree $S$ of $2^{<\omega}$ of measure at least $q$ and an $e \in \omega$ such that $\Phi_e^A$ is a path through $T$ for every path $A$ through $S$.
\end{proposition}

\begin{proof}
Suppose otherwise and fix any computable sequence $\seq{T_i : i \in \omega}$ of (indices for) subtrees of $2^{<\omega}$ of positive measure. We build a single tree $S$ of positive measure, every path through which computes a sequence of sets $\seq{A_i : i \in \omega}$ such that each $A_i$ is a path through $T_i$. In particular, every 1-random set computes such a sequence. Of course, this contradicts the proof of Theorem~\ref{T:Equiv_SeqWWKL_WKL}, as it follows from what is shown there that there exists a sequence of trees for which only sets of PA degree can compute a sequence of paths, but not every $1$-random computes a set of PA degree.

We obtain $S$ by interleaving the members of a new sequence $\seq{S_i : i \in \omega}$ of subtrees of $2^{<\omega}$, constructed inductively as follows. By adding a tree to $\seq{T_i : i \in \omega}$ if necessary, we may assume $\mu(T_0) < 1$, and fix a positive rational number $r$ with $\mu(T_0) < r < 1$. Define $S_0 = T_0$, choose $q_0 < 1$ with $r < q_0$, and let $e_0$ be an index for the identity reduction. Now suppose we have defined $S_i$, $q_i$, and $e_i$. Choose a rational $q_{i+1} < 1$ such that $\prod_{j \leq i+1} q_j \geq r$, which we may assume exists by induction. Let $S_{i+1}$ and $e_{i+1}$ be as given by the hypothesized effective procedure in the statement, with $T = T_{i+1}$ and $q = q_{i+1}$.

Clearly, the resulting sequences $\seq{S_i : i \in \omega}$ and $\seq{e_i : i \in \omega}$ are computable. It follows that $S$ is computable, and by construction, $\mu(S) = \prod_{i \in \omega} \mu(S_i) \geq r > 0$. Now suppose $B$ is any path through $S$. By undoing the interleaving process along $B$, we computably define a sequence $\seq{B_i : i \in \omega}$ such that each $B_i$ is a path through $S_i$. Setting $A_i = \Phi_{e_i}^{B_i}$ for each $i$, it follows that $\seq{A_i: i \in \omega}$ is the desired sequence of paths through the $T_i$.
\end{proof}

The preceding results inspire the following restriction of $\WWKL$. Let $q < 1$ be a positive rational.

\begin{statement}[$q\textrm{-}\WWKL$]\label{S:qWWKL}
Every subtree $T$ of $2^{<\omega}$ such that
\[
\frac{|\{\sigma \in 2^n : \sigma \in T\}|}{2^n} \geq q
\]
for all $n$ has an infinite path.
\end{statement}

\noindent Note that Proposition~\ref{P:uniform_trees} can be formalized to show that $\RCA \vdash \WWKL \leftrightarrow q\textrm{-}\WWKL$, for each $q$. We conclude this section with the following contrasting result.

\begin{proposition}\label{P:qWWKL}
For all positive rationals $p < q < 1$, $p\textrm{-}\WWKL \nred q\textrm{-}\WWKL$.
\end{proposition}

\begin{proof}
Suppose not, and let $\Phi$ and $\Psi$ witness a Weihrauch reduction from $p\textrm{-}\WWKL$ to $q\textrm{-}\WWKL$. We build a computable tree $T$ of measure at least $p$ such that $\Phi(T)$ has measure less than $q$, and thus obtain the desired contradiction. Intuitively, we use the fact that $\Psi$ must take paths through $\Phi(T)$ to paths through $T$ to successively cut down larger and larger portions of $\Phi(T)$ by cutting down larger and larger portions of~$T$. Although this results in the measures of both trees becoming smaller, we will only cut down each tree finitely many times, and we will be able to control for how much of the measure of $T$ is left.

We shall regard each partial computable function as defining an initial segment of a computable subtree of $2^{<\omega}$, with each new convergence giving an entire new level of the tree, and only strings of maximal length at the previous level being extended. Then $\Phi$ in the construction can be viewed as a monotone map between such initial segments. This will ensure that the construction of $T$ will be uniform, and so by the recursion theorem, we can fix an index for it ahead of time. This permits us the convenience of not needing to consider $T$ in the oracle for $\Psi$, by replacing that functional, if necessary, by $\widehat{\Psi}(X) = \Psi(T,X)$.

\medskip
\noindent \emph{Construction.} Fix a positive number $a$ such that $2^{-a} < q - p$. At stage $s$ of the construction we shall define $T_s = T \cap 2^{\leq s}$, starting with $T_0 = \{\emptyset\}$. That is, $T_s$ will have height $s$. Let $n_s$ be the height of $\Phi(T_s)$, and assume without loss of generality that $n_s \leq s$ for all $s$ and that $\Phi(T_0) = \{\emptyset\}$.

At stage $s+1$, choose the least $a$ many numbers $x_0 < \cdots < x_{a-1}$ that we have not yet acted for, as defined below. Assume inductively that for each $\alpha \in 2^a$ there is a string $\sigma \in T_s$ of length $s$ with $\sigma(x_j) = \alpha(j)$ for all $j$ with $x_j < s$. We consider two cases.

\medskip
\noindent \emph{Case 1.} If any of the following apply:
\begin{itemize}
\item $\Phi(T_s)$ contains fewer than $2^{n_s}q$ many strings of length $n_s$;
\item $x_{a-1} \geq s$;
\item $x_{a-1} < s$ but $\Psi(\tau)(x_j) \uparrow$ for some $\tau \in \Phi(T_s)$ of length $n_s$ and some $j < a$;
\end{itemize}
then we obtain $T_{s+1}$ from $T_s$ by adding $\sigma 0$ and $\sigma 1$ for each $\sigma \in T_s$ of length $s$.

\medskip
\noindent \emph{Case 2.} Otherwise, choose $\alpha \in 2^a$ so that $\Psi(\tau)(x_j) \downarrow = \alpha(j)$ for all $j < a$ for at least $2^{-a}$ many strings $\tau \in \Phi(T_s)$ of length $n_s$. Then, we obtain $T_{s+1}$ from $T_s$ by adding $\sigma 0$ and $\sigma 1$ for each $\sigma \in T_s$ of length $s$ with $\sigma(x_j) \neq \alpha(j)$ for some $j < a$. Say we have \emph{acted} for $x_0,\ldots,x_{a-1}$.

\medskip
\noindent \emph{Verification.}
Clearly, $T$ is a computable subtree of $2^{<\omega}$. Note that the measure of $T$ is cut down only when the construction enters Case 2, at which point it is cut down by a factor of precisely $2^{-a}$. Likewise, whenever the construction enters Case 2, the measure of $\Phi(T)$ is cut down by at least a factor of $2^{-a}$. We claim there is a stage $s$ such that $\Phi(T_s)$ contains fewer than $2^{n_s}q$ many strings of length $s$, so that the measure of $\Phi(T)$ is less than $q$. Fix the least such $s$. Then as $2^{-a} < q - p$, it follows that $T_s$ contains at least $2^{s}p$ many strings of length $s$. But the construction can never enter Case 2 at any stage after $s$, so the measure of $T$ is at least $p$.

It thus remains only to prove the claim. To this end, let $t$ be any stage such that $\Phi(T_t)$ contains at least $2^{n_t}q$ many strings of length $n_t$. Fix the least $x_0 < \cdots < x_{a-1}$ not yet acted for prior to stage $t+1$. For each path $B$ through $\Phi(T)$, we have that $\Psi(B)(x_j) \downarrow$ for all $j < a$, so by compactness, there is an $s > \max\{t,x_{a-1}\}$ such that $\Psi(\tau)(x_j) \downarrow$ for all $j < a$ and all $\tau \in \Phi(T_s)$ of length $n_s$. Fix the least such $s$. Then the construction never enters Case 2 strictly between stages $t$ and $s$, so $T_s$ contains at least $2^{n_s}q$ many strings of length $n_s$, so Case 2 applies at stage $s$. Hence, we have shown that the construction continues to enter Case 2 until the measure of $\Phi(T)$ has been sufficiently cut down, from which the claim follows.
\end{proof}

\section{The Thin Set Theorem}

For all $n \geq 1$ and $k \in \{2,3,4,\dots,\omega\}$, say that a subset $S$ of $\omega$ is \emph{thin} for a coloring $f \colon [\omega]^n \to k$ if there exists a $c < k$ such that $f(\tuple{x}) \neq c$ for all $\tuple{x} \in [S]^n$. In this section, we shall concentrate on the following combinatorial principle, known as the Thin Set Theorem.

\begin{statement}[$\TS^n_k$]
Let $n \geq 1$ and let $k \in \{2,3,4,\dots,\omega\}$.  Every $f \colon [\N]^n \to k$ admits an infinite thin set.
\end{statement}

\noindent The statement $\TS^n_{\omega}$ is the usual Thin Set Theorem as studied in~\cite{CGHJ-2005}.\footnote{This should not be confused with the principle $\TS^n_{<\infty}$, which, by analogy with Ramsey's theorem, should be defined as $(\forall k \geq 2)~\TS^n_k$. By contrast, $\TS^n_{\omega}$ is the statement of the Thin Set Theorem for colorings $f : [\N]^n \to \omega$, i.e.,~colorings employing infinitely many colors. Using Proposition~\ref{P:TS_trivial_implication}, is not difficult to see that $\TS^n_{<\infty}$ is equivalent to $\TS^n_2$ under strong Weihrauch\ reducibility.}  Note that $\TS^n_2$ is logically equivalent to $\RT^n_2$, i.e., the thin sets for $2$-colorings are precisely the homogeneous sets. Likewise, observe that whereas $\RT^n_1$ is plainly true, $\TS^n_1$ is not even defined above, as it would be plainly false.

Implications between versions of the Thin Set Theorem for different numbers of colors go opposite the way they do for Ramsey's theorem.
For the purpose of viewing $\TS^n_k$ as a multi-valued function, it is important that a solution to an instance of $\TS^n_k$ includes which color is omitted by the thin set since there is no uniformly computable way to recover that information from the thin set alone.

\begin{proposition}\label{P:TS_trivial_implication}
Let $n \geq 1$.
\begin{enumerate}
\item If $j,k \geq 2$ with $j < k$, then $\TS^n_k \sred \TS^n_j$.
\item If $j,k \geq 2$ with $j < k$, then $\RCA \vdash \TS^n_j \rightarrow \TS^n_k$.
\item If $j \geq 2$, then $\TS^n_{\omega} \sred \TS^n_j$.
\item If $j \geq 2$, then $\RCA \vdash \TS^n_j \rightarrow \TS^n_{\omega}$.
\end{enumerate}
\end{proposition}

\begin{proof}
We prove $(1)$ and $(2)$ (the argument is uniform and can easily be formalized in $\RCA$).  Let $j,k \geq 2$ with $j < k$.  Fix $f \colon [\omega]^n \to k$. Define $g \colon [\omega]^n \to j$ by letting
\[
g(\tuple{x}) =
\begin{cases}
f(\tuple{x}) & \textrm{if } f(\tuple{x}) < j-1,\\
j-1 & \textrm{otherwise}
\end{cases}
\]
for all $\tuple{x} \in [\omega]^n$. Now suppose $S \subseteq \omega$ is an infinite thin set for $g$, say with $c < j$ such that $g(\tuple{x}) \neq c$ for all $\tuple{x} \in [S]^n$. If $c < j-1$, then $f(\tuple{x}) \neq c$ for all $\tuple{x} \in S$, while if $c = j-1$ then $f(\tuple{x}) < j-1$ for all such $\tuple{x}$, so in particular $f(\tuple{x}) \neq j-1 = c$. Either way, $c$ witnesses that $S$ is an infinite thin set for $f$.  The proof of (3) and (4) similarly proceeds by collapsing all colors greater then $j-1$ to be $j-1$.
\end{proof}

Thus, we have the following chain for any $n$:
\[
\TS_{\omega}^n \sred \dots \sred \TS_4^n \sred \TS_3^n \sred \TS_2^n = \RT_2^n \sred \RT_3^n \sred \RT_4^n \sred \dots
\]
By Theorem~\ref{T:Ramsey_non-uniform}, none of the reductions to the right of the equals sign reverse. We shall see in Theorem~\ref{T:TS1_nonsquashing} that the same is true of the left side when $n=1$.

\subsection{General reverse mathematics results}

Before discussing uniform implications and sequential forms, we prove several results about the principles $\TS^n_k$. General questions about the strength of $\TS^n_k$ were asked by J.~Miller at the \emph{Reverse Mathematics: Foundations and Applications Workshop} in Chicago in November, 2009. Another recent investigation of these principles appears in Wang \cite{Wang-TA2}.

\begin{proposition}\label{P:Step}
For each $m,n,k \geq 1$, we have $\RCA \vdash \TS^{mn+1}_{k^n} \rightarrow \TS^{m+1}_{k}$.
\end{proposition}

\begin{proof}
The result is trivial for $n = 1$, so we may assume $n \geq 2$. Let $f \colon [\N]^{m+1} \to k$ be a coloring. Define $g \colon [\N]^{mn+1} \to k^n$ by
\[
g(x,\tuple{y}_0,\ldots,\tuple{y}_{n-1}) = \seq{f(x,\tuple{y}_0),\ldots,f(x,\tuple{y}_{n-1})}
\]
for all $x \in \omega$ and $\tuple{y}_0,\ldots,\tuple{y}_{n-1} \in [\omega]^m$ with $x < \tuple{y}_0 < \cdots < \tuple{y}_{n-1}$.

Suppose $H$ is an infinite set that avoids the color $\seq{a_0,\ldots,a_{n-1}} < k^n$ for the coloring $g$. Choose the greatest $i<n$ for which there are infinitely many $x \in H$ such that
\begin{equation}\label{E:initial_agreement}
f(x,\tuple{y}_0) = a_0, \ldots, f(x,\tuple{y}_i) = a_i
\end{equation}
for some $\tuple{y}_0,\ldots,\tuple{y}_i \in [H]^m$ with $x < \tuple{y}_0 < \cdots < \tuple{y}_i$. By assumption on the color avoided by $H$, it must be that $i < n-1$.

By choice of $i$, we can remove finitely many elements from $H$ if necessary to ensure that if $x < \tuple{y}_0 < \dots < \tuple{y}_i$ satisfy \eqref{E:initial_agreement} above, then there is no $\tuple{y} > \tuple{y}_i$ such that $f(x,\tuple{y}) = a_{i+1}$. Let $H'$ be $H$ with these finitely many elements deleted.

Now using $\Delta^0_1$ comprehension, we can define a sequence $\seq{x_j : j \in \omega}$ of elements of $H'$ so that for each $j$, \eqref{E:initial_agreement} holds for some $\tuple{y}_0,\ldots,\tuple{y}_i \in [H']^m$ with
\begin{equation}\label{E:increasing}
x_j < \tuple{y}_0 < \cdots < \tuple{y}_i < x_{j+1}.
\end{equation}
Let $R \subseteq H'$ be the range of this sequence, which exists because the sequence is increasing. We claim that $R$ avoids the color $a_{i+1}$ for $f$. Indeed, suppose $f(x,\tuple{y}) = a_{i+1}$ for some $x \in R$ and $\tuple{y} \in [R]^m$ with $x < \tuple{y}$. Let $\tuple{y}_0,\ldots,\tuple{y}_i \in [H']^m$ be the witnesses for having chosen $x$ to belong to our sequence. Then by \eqref{E:increasing}, it follows that $\tuple{y}_i < \tuple{y}$, which contradicts the definition of $H'$.
\end{proof}

\noindent
Setting $k = 2$, we obtain:

\begin{corollary}\label{C:TS>RT22}
For each $m,n \geq 1$, we have $\RCA \vdash \TS^{mn+1}_{2^n} \rightarrow \RT^{m+1}_2$.
\end{corollary}

\noindent
The proof of Proposition~\ref{P:Step} is not entirely uniform.
Indeed, one assumes $\Sigma^0_2$-induction in order to prove that $\forall n(\TS^{mn+1}_{k^n} \to \TS^{m+1}_k).$
Similarly, the proof does not give a Weihrauch reduction of $\TS^{m+1}_k$ to $\TS^{mn+1}_{k^n}$.

Since $\RT^3_2$ is equivalent to arithmetic comprehension, we also get that $\TS^{2n+1}_{2^n}$ implies arithmetic comprehension for each $n \in \omega$.  However, we can do better by carefully choosing the coloring.

\begin{proposition}\label{P:TS>ACA}
For each $n \geq 1$, we have $\RCA \vdash \TS^{n+2}_{2^n} \rightarrow \mathsf{ACA}$.
\end{proposition}

\begin{proof}
We will show how to reduce finding the range of an injection $f : \N\to\N$ to an instance $g : [\N]^{n+2} \to 2^n$ of $\TS^{n+2}_{2^n}$. Namely, for each $i < n$, define the coloring $g_i : [\N]^{n+2} \to 2$ by
   \[g_i(x_0,\dots,x_{n+1}) = \begin{cases} 1 & (\exists z)[x_i < z < x_{i+1} \wedge f(z) < x_0], \\ 0 & (\forall z)[x_i < z < x_{i+1} \rightarrow f(z) \geq x_0], \end{cases}\]
  and let $g = \seq{g_0,\ldots,g_{n-1}}$.
  Let $H$ be an infinite set that avoids at least one color $b = \seq{b_0,\ldots,b_{n-1}} < 2^n$.
  Let $m<n$ be the largest index for which there are $x_0 < \cdots < x_{n+1}$ in $H$ with $g_i(x_0,\dots,x_{n+1}) = b_i$ for all $i < m$, and assume without loss of generality that such $x_0,\ldots,x_{n+1}$ can be found in every tail of $H$.


  To determine whether some number $y$ is in the range of $f$, choose some elements $x_0 < \cdots < x_n$ of $H$ with $y < x_0$ and $g_i(x_0,\dots,x_n) = b_i$ for $i < m$.
  Note that $f(z) \geq x_0$ for all $z > x_m$, otherwise we could pick $y_0 = x_0 < \cdots < y_m = x_m < y_m < \cdots < y_{n+1}$ in $H$ to realize at least $m+1$ bits of $b$, contradicting the choice of $m$.
  Therefore, $y$ is in the range of $f$ if and only if $y \in \set{f(0),\dots,f(x_m)}$.
\end{proof}

By contrast, Wang \cite[Theorem 3.1]{Wang-TA2} has shown that for every $n$, there is a $k$ such that $\TS^n_k$ does not imply $\mathsf{ACA}$ over $\RCA$. Thus, the number of colors above is important.

\begin{proposition}[$\RCA$]\label{P:TS>RT1}
For all $n \geq 1$, we have $\RCA \vdash \TS^{n+1}_{3^n} \rightarrow \RT^1_{<\infty}$.
\end{proposition}

\begin{proof}
  By Proposition~\ref{P:Step}, it suffices to show that $\TS^2_3$ implies $\RT^1_{<\infty}$.
  Given $f \colon \N \to k$, define $g \colon [\N]^2 \to 3$ by
  \[
  g(x,y) =
  \begin{cases}
  0 & \text{if $f(x) = f(y)$,} \\
  1 & \text{if $f(x) > f(y)$,}\\
  2 & \text{if $f(x) < f(y)$,}
  \end{cases}
  \]
  for all $x < y$.
  Suppose that $H$ is an infinite set that avoids one of the three colors.

  Note that $H$ cannot avoid the color $0$, since otherwise the restriction of $f$ to $H$ would be an injection, which is impossible since $H$ is infinite. So suppose $H$ avoids the color $1$, so that the restriction of $f$ to $H$ is then non-decreasing. Any bounded non-decreasing function on an infinite set eventually stabilizes to a maximal value $m$. Then $f^{-1}(m)$ is an infinite homogeneous set for $f$. The case when $H$ avoids color $2$ is symmetric.
\end{proof}

\subsection{Results for triples}

In order to properly state the strong Weihrauch\ reduction form of the next few results, we introduce an operation on $\Pi^1_2$ statements that essentially corresponds to composition of partial multi-valued functions in the context of computable analysis. Given $\Pi^1_2$ statements $\mathsf{P}$ and $\mathsf{Q}$ and a Turing functional $\Theta$, we define $\mathsf{Q} \after \mathsf{P}$ to be the $\Pi^1_2$ principle whose instances are instances $A$ of $\mathsf{P}$ and whose solutions are pairs $\seq{B,C}$ where $B$ is a $\mathsf{P}$-solution for $A$ and $C$ is a $\mathsf{Q}$-solution for the instance $\Theta(A,B)$ of $\mathsf{Q}$. In other words, $\mathsf{Q} \after \mathsf{P}$ first takes an instance $A$ of $\mathsf{P}$ and seeks a solution $B$, then uses $A$ and $B$ to construct an instance $\Theta(A,B)$ of $\mathsf{Q}$ and seeks a solution $C$. The Turing functional $\Theta$ is only a matter of convenience in order to translate the output of $\mathsf{P}$ into an input for $\mathsf{Q}$ and it is usually obvious from the context what $\Theta$ needs to be. (The notion here essentially corresponds to function composition from the context of computable analysis, although because of the use of the functional $\Theta$, and because solutions here are pairs, the two are not formally the same. Compare this with the compositional product in the Weihrauch lattice, as defined in \cite[Definition 4.1]{BGM-2012}.)

To illustrate this definition, consider for example the implication of $\RT^2_2$ by $\SRT^2_2 + \COH$ from Cholak, Jockush and Slaman~\cite[Theorem 12.5]{CJS-2001}. Their proof that $\SRT^2_2 + \COH$ implies $\RT^2_2$ breaks into three steps, as follows. A given coloring $f : [\omega]^2 \to 2$ is first transformed into the instance $\seq{R_x : x \in \omega}$ of $\COH$ given by $R_x = \{y : y > x \land f(x,y) = 0\}$; a solution $C = \{c_0 < c_1 < \ldots \}$ to this instance is then transformed, along with $f$, into the stable coloring $g : [\omega]^2 \to 2$ given by $g(x,y) = f(c_x,c_y)$; and finally, a homogeneous set $G$ for $g$ is transformed into the homogeneous set $H = \{c_x : x \in G\}$ for $f$. Thus, by letting $\Theta$ be the functional that defines $g$ from $C$ and $f$ as here, we see that this argument corresponds to a reduction $\RT^2_2 \sred \SRT^2_2 \after \COH$. The key property used here is that the restriction of a computable coloring to a cohesive set is stable, and thus keeping the output $C$ of $\COH$ in the composition is essential.

We know that $\TS^3_2$ implies arithmetic comprehension, and Corollary~\ref{C:TS>RT22} shows that $\TS^3_4$ implies $\RT^2_2$.
This leaves a gap around $\TS^3_3$. Wang \cite[Corollary 3.2]{Wang-TA2} has shown that $\TS^3_3$ does not imply $\mathsf{ACA}$. The next result gives a little more information. 

\begin{proposition}\label{P:TS33}
$\RCA \vdash \TS^3_3 \rightarrow \RT^2_{<\infty}$.
\end{proposition}
  
\begin{proof}
Let $f \colon [\N]^2\to k$ be any finite coloring. Let $g \colon [\N]^3\to 3$ be defined by
\[
g(x,y,z) = |\set{f(x,y),f(x,z),f(y,z)}| - 1.
\]  
By $\TS^{3}_{3}$ there is an infinite set $H$ that omits one of the three possible colors.
  
  Since every infinite set contains at least one homogeneous triangle for $f$, the set $H$ cannot omit color $0$ for $g$.
  If $H$ omits color $1$ for $g$, then pick $x \in H$ and consider the sets $H_i = \set{y \in H : y > x \land f(x,y) = i}$ for $i < k$.
  Since $H$ omits triples which take exactly two $f$-colors, each $H_i$ is homogeneous with color $i$.
  By $\mathsf{B}{\Pi^0_1}$, which follows from $\RT^2_2$ and hence from $\TS^3_3$, one of these sets must be infinite. Thus we have an infinite homogeneous set for $f$.

  The only remaining case is when $H$ omits color $2$ for $g$.
  In that case, consider the coloring \[h(x,y,z) = \begin{cases} 0 & \text{if $f(y,z) = f(x,y) = f(x,z)$,} \\ 1 & \text{if $f(y,z) \neq f(x,y) = f(x,z)$,} \\ 2 & \text{if $f(x,z) \neq f(x,y) = f(y,z)$,} \\ 3 & \text{if $f(x,y) \neq f(x,z) = f(y,z)$,} \end{cases}\] where $x < y < z$ are elements of $H$.
  Since $H$ omits color $2$ for $g$, these four cases are exhaustive.
  By $\TS^3_3$, there is an infinite set $G \subseteq H$ that omits two colors for $h$.

  Since every infinite set contains a homogeneous triangle, color $0$ cannot be among the colors omitted by $G$.
  If $1$ is among the two colors omitted by $G$, then pick $x \in G$ and consider the sets $G_i = \set{y \in G : y > x \land f(x,y) = i}$ for $i < k$.
  Since $G$ omits color $1$ for $h$, each $G_i$ is homogeneous of color $i$.
  By $\mathsf{B}{\Pi^0_1}$, one of these sets $G_i$ must be infinite and thus we have an infinite homogeneous set for $f$.

  The only remaining case is when $G$ omits both colors $2$ and $3$ for $h$.
  In that case, $G$ is $\min$-homogeneous for $f$, i.e., $f(x,y) = f(x,z)$ for all $x < y < z$ in $G$.
  We may then unambiguously define the coloring $\bar{f} \colon \N \to k$ by $\bar{f}(x) = f(x,y)$ where $x < y$ in $G$.
  By $\mathsf{B}{\Pi^0_1}$, one of the sets $G_i = \set{x \in G : \bar{f}(x) = i}$ must be infinite, and this $G_i$ is an infinite homogeneous set for $f$ of color $i$.
\end{proof}

\noindent
Because this argument compounds multiple uses of $\TS^3_3$ and its consequences, Proposition~\ref{P:TS33} does not lead a direct reduction of $\RT^2_{<\infty}$ to $\TS^3_3$. Carefully going through the proof and using the $\bullet$ composition described above, we obtain the following cumbersome reduction: $$\RT^2_k \sred \RT^1_k\bullet\TS^3_3\bullet\TS^3_3\bullet\RT^1_k\bullet\TS^3_3.$$

It is also unclear how strong $\TS^3_k$ is for $k \geq 4$.
The next three results give some non-trivial lower bounds for $k = 6,7,8$.

For these results, we use the following related notions:

\begin{definition}\label{D:transitive_hereditary_colorings}
A coloring $f \colon [\N]^2 \to k$ is \emph{transitive} if for all $i < k$ and all $x < y < z$, whenever $f(x,y) = f(y,z) = i$ then $f(x,z) = i$. A coloring is \emph{semi-transitive} if this property holds for all but possibly one $i < k$.

A coloring $f \colon [\N]^2 \to k$ is \emph{semi-hereditary} if for all $i < k$ except possibly one, whenever $x < y < z$ and $f(x,z) = f(y,z) = i$ then $f(x,y) = i$.

A coloring $f \colon [\N]^2 \to k$ is \emph{semi-trivial} if for all $i < k$ except possibly one, the set $\{y \in \N : x < y \land f(x,y) = i\}$ is homogeneous for $f$ for each $x$ (in a possibly different color).
\end{definition}

These are associated with restrictions of $\RT^2_2$.

\begin{statement}[$\ADS$]
Every transitive coloring $f \colon [\N]^2 \to 2$ has an infinite homogeneous set.
\end{statement}
\begin{statement}[$\CAC$]
Every semi-transitive coloring $f \colon [\N]^2 \to 2$ has an infinite homogeneous set.
\end{statement}
\begin{statement}[$\CHO$]
Every semi-hereditary coloring $f \colon [\N]^2 \to 2$ has an infinite homogeneous set.
\end{statement}
\begin{statement}[$\PS$]
Every semi-trivial coloring $f \colon [\N]^2 \to 2$ has an infinite homogeneous set.
\end{statement}

The restrictions $\CAC$ and $\ADS$ were studied by Hirschfeldt and Shore~\cite{HS-2007}, who showed that $\ADS$ is implied by $\CAC$ over $\RCA$, and that both are strictly weaker than $\RT^2_2$. Recently, Lerman, Solomon and Towsner~\cite[Section 2]{LST-2013} have shown that $\ADS$ does not imply $\CAC$ over $\RCA$. (The usual definitions of these principles, as given in Section 1 of~\cite{HS-2007}, are equivalent to the ones above by Theorems 5.2 and 5.3 of~\cite{HS-2007}, respectively.) The restriction $\CHO$ was studied by Dorais (unpublished), who showed that it follows from $\ADS$. It is unknown whether $\CHO$ implies $\ADS$ over $\RCA$.
The last restriction, $\PS$, is equivalent to the infinite pigeonhole principle $\RT^1_{<\infty}$ over $\RCA$. However, the proof of $\PS$ from $\RT^1_{<\infty}$ is not uniform and hence there does not appear to be a strong Weihrauch equivalence between $\PS$ and $\RT^1_{<\infty}$. $\CAC$, $\ADS$, $\CHO$ all imply $\PS$ over $\RCA$ and none of those implications reverse.

We now prove a number of implications between $\Pi^1_2$ principles in $\RCA$ that can also be presented as strong Weihrauch reductions between compositions of $\Pi^1_2$ principles.
To properly state the relevant reductions, we need the \emph{alternative product} $[\mathsf{P},\mathsf{Q}]$ of $\Pi^1_2$ principles $\mathsf{P}$ and $\mathsf{Q}$.
An instance of $[\mathsf{P},\mathsf{Q}]$ is a either a pair $\seq{0,A}$ where $A$ is an instance of $\mathsf{P}$ or a pair $\seq{1,B}$ where $B$ is an instance of $\mathsf{Q}$; a corresponding solutions are, respectively, solutions to the instance $A$ of $\mathsf{P}$ or solutions to the instance $B$ of $\mathsf{Q}$.
This is indeed a product since $\mathsf{P}, \mathsf{Q} \sred [\mathsf{P},\mathsf{Q}]$ but it is not always equivalent to the parallel product $\seq{\mathsf{P},\mathsf{Q}}.$
In fact, $[\mathsf{P},\mathsf{Q}]$ is the least upper bound of $\mathsf{P}$ and $\mathsf{Q}$ in the strong Weihrauch preordering. In particular, $[\mathsf{P},\mathsf{P}]$ is always strong Weihrauch equivalent to $\mathsf{P}$.
(See Blass~\cite{Blass-1995} for a discussion of these two products in a broader setting. The alternative product is called the co-product in the Weihrauch lattice, and was originally introduced in this context by Pauly \cite{Pauly-2010}.)


\begin{proposition}\label{P:TS38CAC>RT22}
$\RT^2_2 \sred [\PS,\CAC,\CHO] \after \TS^3_8$ and $\RCA \vdash (\TS^3_8 + \CAC) \rightarrow \RT^2_2$.
\end{proposition}

\begin{proof}
  We argue in $\RCA$. Let $f \colon [\N]^2 \to 2$ be a coloring.
  Define the coloring $g \colon [\N]^3 \to 8$ by \[g(x,y,z) = \seq{f(y,z),f(x,z),f(x,y)},\] for all $x < y < z$.
  By $\TS^3_8$, we know that there is an infinite set $H$ that avoids one of the eight possible colors for $g$.
  
  The proof now divides into cases according to which color is avoided. Call this color $c$.

  \par\medskip\noindent\emph{Case 1.} 
  If $c = \seq{0,0,0}$ or $c = \seq{1,0,0}$, then for every $x \in H$, the set $H_x = \set{y \in H : y > x \land f(x,y) = 0}$ is homogeneous for $f$ (with color $1$ or $0$, respectively). 
  If $c = \seq{1,1,1}$ or $c = \seq{0,1,1}$, then for every $x \in H$, the set $H_x = \set{y \in H : y > x \land f(x,y) = 1}$ is homogeneous for $f$ (with color $0$ or $1$, respectively).
  In all these cases, appliying $\PS$ gives an infinite homogeneous set for $f$.

  \par\medskip\noindent\emph{Case 2.} If $c = \seq{0,1,0}$ or $c = \seq{1,0,1}$, then $f$ is semi-transitive on $H$. Applying $\CAC$ gives an infinite homogeneous set for $f$.

  \par\medskip\noindent\emph{Case 3.} If $c = \seq{0,0,1}$ or $c = \seq{1,1,0}$, then $f$ is semi-hereditary on $H$. Applying $\CHO$ gives an infinite homogeneous set for $f$.
\end{proof}

\begin{proposition}\label{P:TS37ADS>RT22}
$\RT^2_2 \sred [\PS,\ADS,\CHO] \after \TS^3_7$ and $\RCA \vdash (\TS^3_7 + \ADS) \rightarrow \RT^2_2$.
\end{proposition}

\begin{proof}
  The construction is basically the same as that of Proposition~\ref{P:TS38CAC>RT22}, with the exception that the dual pair of colors $\seq{0,1,0}$ and $\seq{1,0,1}$ are merged into one.
  Any infinite set $H$ that avoids both of these colors is transitive for $f$, so $\ADS$ suffices to give an infinite homogeneous set for $c$.
\end{proof}

\begin{proposition}\label{P:TS36CHO>RT22}
$\RT^2_2 \sred [\PS,\CHO] \after \TS^3_6$ and $\RCA \vdash (\TS^3_6 + \CHO) \rightarrow \RT^2_2$.
\end{proposition}

\begin{proof}
  The construction is basically the same as that of Proposition~\ref{P:TS38CAC>RT22}, with the exception that the pairs of colors $\seq{0,1,0}$, $\seq{1,1,0}$ are merged into one, and similarly for the dual pair $\seq{1,0,1}$, $\seq{0,0,1}$.
  Then $\CHO$ suffices to give an infinite homogeneous set for $f$ these two cases.
\end{proof}

\subsection{Sequential forms}\label{Sec:sequential_forms}

From $\TS^n_k$ we can, in accordance with Statement~\ref{S:sequential_form}, form the sequential version $\Seq\TS^n_k$.  Surprisingly, the sequential forms of these weaker thin set principles can still code $\emptyset^{(n)}$ just as $\Seq\RT^n_k$ did in Lemma~\ref{L:jump_coding}.  We need the following theorem of Kummer.

\begin{theorem}[Kummer~\cite{Kummer-1992}, p.~678]\label{T:Kummer}
Fix $k \geq 2$, and let $A,B \subseteq \omega$ be arbitrary. Suppose $g$ is a computable function such that, for all $\tuple{x} \in [\omega]^{k-1}$, the $B$-c.e.\ set $W^B_{g(\tuple{x})}$ is a proper subset of $k$ and contains $| \tuple{x} \cap A|$. Then $A$ is computable in $B$.
\end{theorem}

\begin{corollary}\label{C:coding_into_seqTSnk}
For all $n \geq 1$ and all $k \geq 2$, there is a computable instance of $\Seq\TS^n_k$, every solution to which computes $\emptyset^{(n)}$.
\end{corollary}

\begin{proof}
The proof is somewhat similar to that of Lemma~\ref{L:jump_coding}, but our argument here is slightly more delicate on account of needing to fit the rather unique conditions of Kummer's theorem. We define a computable sequence $\seq{f_{\tuple{x}} : \tuple{x} \in [\omega]^{k-1}}$ of $k$-colorings of $[\omega]^n$ to serve as the desired instance of $\Seq\TS^n_k$. Let $h$ be a $\{0,1\}$-valued computable function such that
\[
\emptyset^{(n)} = \{i \in \omega : \lim_{y_0} \cdots \lim_{y_{n-1}} h(i,y_0,\ldots,y_{n-1}) = 1\},
\]
and for each $\tuple{x} \in [\omega]^{k-1}$ set
\[
f_{\tuple{x}}(\tuple{y}) = |\{i \in \tuple{x} : h(i,y_0,\ldots,y_{n-1}) = 1\}|
\]
for all $\tuple{y} = \seq{y_0,\ldots,y_{n-1}} \in [\omega]^n$.

Suppose we are given $\vec{H} = \seq{H_{\tuple{x}} : \tuple{x} \in [\omega]^{k-1}}$ such that each $H_{\tuple{x}}$ is an infinite thin set for $f_{\tuple{x}}$. That is, $\vec{H}$ is a solution to the instance $\seq{f_{\tuple{x}} : \tuple{x} \in [\omega]^{k-1}}$. Let $g$ be a computable function such that for all $\tuple{x} \in [\omega]^{k-1}$,
\[
W^{\vec{H}}_{g(\tuple{x})} = \{ c < k : (\exists \tuple{y} \in [H_{\tuple{x}}]^n)[f_{\tuple{x}}(\tuple{y}) = c] \}.
\]
Since it is thin, $H_{\tuple{x}}$ necessarily avoids some $c < k$, so $W^{\vec{H}}_{g(\tuple{x})}$ is a proper subset of $k$. We claim that $|\tuple{x} \cap \emptyset^{(n)}| \in W^{\vec{H}}_{g(\tuple{x})}$, whence it will follow by Theorem~\ref{T:Kummer} that $\emptyset^{(n)} \leq_T \vec{H}$, as desired.

To prove the claim, fix $\tuple{x} \in [\omega]^{k-1}$. For each $i \in \tuple{x}$, we have that
\[
(\exists s_0)(\forall y_0 > s_0) \cdots (\exists s_{n-1})(\forall y_{n-1} > s_{n-1})~[h(i,y_0,\ldots,y_{n-1}) = \emptyset^{(n)}(i)]
\]
by definition of the limit. Let $w_0,\ldots,w_{n-1}$ be Skolem functions for this definition, so that for each $i$,
\[
(\forall y_0 > w_0(i)) \cdots (\forall y_{n-1} > w_{n-1}(i,y_0,\ldots,y_{n-2}))~[h(i,y_0,\ldots,y_{n-1}) = \emptyset^{(n)}(i)].
\]
We define a sequence $s_0 < y_0 < s_1 < y_1 \cdots < s_{n-1} < y_{n-1}$ with each $s_j \in \omega$ and each $y_j \in H_{\tuple{x}}$, as follows. Let $s_0 = \max_{i \in \tuple{x}} \{w_0(i)\}$, and suppose $s_j$ has been defined for some $j < n$. Let $y_j$ be the least element of $H_{\tuple{x}}$ greater than $s_j$, and if $j < n-1$, let $s_{j+1} = \max_{i \in \tuple{x}}\{ w_{j+1}(i,z_0,\ldots,z_j) : (\forall k \leq j)[z_k \leq y_k]\}$.

By construction, we have that $h(i,y_0,\ldots,y_{n-1}) = \emptyset^{(n)}(i)$ for all $i \in \tuple{x}$. Hence, by definition, $f_{\tuple{x}}(y_0,\ldots,y_{n-1}) = |\tuple{x} \cap {\emptyset^{(n)}}|$. But as the $y_j$ were all chosen from $H_{\tuple{x}}$, this means that $|\tuple{x} \cap \emptyset^{(n)}|$ is not a color omitted by $H_{\tuple{x}}$. This is what was to be shown.
\end{proof}

\noindent
Unfortunately, the proof of Kummer's theorem is not uniform and hence Corollary~\ref{C:coding_into_seqTSnk} does not lead to a Weihrauch reduction.

Since each $\TS^n_k$ is clearly total and has finite tolerance, we may now apply the Squashing Theorem to obtain the following consequence.

\begin{corollary}
For all $n \geq 1$ and $j,k \geq 2$, we have $\seq{\TS^n_k,\TS^n_j} \nred \TS^n_j$.
\end{corollary}

\begin{proof}
By Proposition~\ref{P:TS_trivial_implication}, $\TS^n_j \sred \TS^n_2 \sred \RT^n_2$. Hence, as described in the proof of Lemma~\ref{L:cant_add_coloring}, every computable instance of $\TS^n_k$ has a solution $H$ with $H' \leq_T \emptyset^{(n)}$.  As there is a computable instance of $\Seq\TS^n_k$ all of whose solutions compute $\emptyset^{(n)}$ by Corollary \ref{C:coding_into_seqTSnk}, it follows that $\Seq\TS^n_k \nred \TS^n_j$. The result follows by applying Theorem~\ref{T:squashing_theorem}.
\end{proof}

We have shown that for each $k \geq 2$, we can code $\emptyset^{(n)}$ into a computable instance of $\Seq\TS^n_k$.  However, $\emptyset^{(n)}$ is not able to solve all computable instances.  We first prove this in the case when $n = 1$.

\begin{theorem}\label{T:diagonalize_TS1k}
For each $k \geq 2$, there exists a computable instance of $\Seq\TS^1_k$ with no $\emptyset'$-computable solution.
\end{theorem}

\begin{proof}
Using the Limit Lemma, we may fix a computable $g \colon \omega^4 \to 2$ such that for every $\Delta^0_2$ set $D \subseteq \omega^2$, there exists an $e \in \omega$ such that:
\begin{itemize}
\item for all $\seq{i,a} \in D$, we have $\lim_s g(e,i,a,s) = 1$;
\item for all $\seq{i,a} \notin D$, we have $\lim_s g(e,i,a,s) = 0$.
\end{itemize}
Concretely, we may let
\[
g(e,i,a,s) =
\begin{cases}
1 & \text{ if } \Phi_{e,s}^{K_s}(i,a)\downarrow \ = 1, \\
0 & \text{ otherwise}.
\end{cases}
\]
We now define our computable instance $\seq{f_i : i \in \omega}$ of $\Seq\TS^1_k$.  We build our sequence so that each $f_i$ is defined independently of the others in such a way that $f_e$ defeats the $e$th potential $\Delta_2^0$ solution $D = \seq{D_i : i \in \omega}$ by ensuring that $D_e$ is not an infinite thin set for $f_e$.  

\medskip
\noindent \emph{Construction.} For a given $i$, we define $f_i(s)$ recursively in stages based on $s$.  Fix $i \in \omega$, and suppose that we are at stage $s$ so that we have defined $f_i(t)$ for all $t < s$.  Let $A_{i,s}$ be the approximation to those elements in the $i${th} column of the $i${th} possible $\Delta_2^0$ set at stage $s$, i.e.,
\[
A_{i,s} = \{b \in \omega : b < s \text{ and } g(i,i,b,s) = 1\}
\]
Let $C_{i,s} = \{f_i(b) : b \in A_{i,s}\}$ be the set of colors used by the elements of this approximation.  We have two cases.

\medskip
\noindent \emph{Case 1.} Suppose that there exists $n < k$ such that $n \notin C_{i,s}$, i.e.,~suppose that some color is not used on the approximation.  We then define $f_i(s)$ to be the least $n < k$ such that $n \notin C_{i,s}$.

\medskip
\noindent \emph{Case 2.} Suppose that $C_{i,s} = \{0,1,\dots,k-1\}$.  For each $n < k$, let $b_{n,s} < s$ be least such that $f_i(b_{n,s}) = n$, i.e.,~$b_{n,s}$ is the first place where color $n$ occurs.  Fix $\ell < k$ such that $b_{\ell,s} = \max\{b_{n,s} : n < k\}$, i.e.,~pick the color whose first occurrence is as late as possible.  Define $f_i(s) = \ell$.  

\medskip
\noindent \emph{Verification.} We now verify that there is no $\Delta_2^0$ solution for $\seq{f_i : i \in \omega}$.  Suppose that $D = \seq{D_i : i \in \omega}$ is a $\Delta_2^0$ set.  Fix $e \in \omega$ such that:
\begin{itemize}
\item for all $\seq{i,b} \in D$, we have $\lim_s g(e,i,b,s) = 1$;
\item for all $\seq{i,b} \notin D$, we have $\lim_s g(e,i,b,s) = 0$.
\end{itemize}
In particular, we have the following:
\begin{itemize}
\item for all $b \in D_e$, we have $\lim_s g(e,e,b,s) = 1$;
\item for all $b \notin D_e$, we have $\lim_s g(e,e,b,s) = 0$.
\end{itemize}
If $D_e$ is finite, then $\seq{D_i : i \in \omega}$ is not a solution to $\seq{f_i : i \in \omega}$ by definition.  Assume then that $D_e$ is infinite.  Let $C_e = \{f_e(b) : b \in D_e\}$ be the set of colors that occur on $D_e$.  We claim that $C_e = \{0,1,\dots,k-1\}$.  Suppose not.  For each $n \in C_e$, let $b_n$ be the least element of $D_e$ such that $f_e(b_n) = n$.  Let $m = \max\{b_n : n \in C_e\}$.  Fix $t > m$ such that the approximation to each element of the $e${th} column below $m$ has settled down, i.e.,~such that:
\begin{itemize}
\item for all $b \in D_e$ with $b \leq m$, we have $g(e,e,b,s) = 1$ whenever $s \geq t$;
\item for all $b \notin D_e$ with $b \leq m$, we have $g(e,e,b,s) = 0$ whenever $s \geq t$.
\end{itemize}
Now take any $s \geq t$.  Notice that $A_{e,s} \cap \{0,1,\dots,m\} = D_e \cap \{0,1,\dots,m\}$, hence $C_e \subseteq C_{e,s}$ and $b_{n,s} = b_n$ for all $n \in C_e$.  Furthermore, if $\ell \notin C_e$ and $b_{\ell,s}$ is defined, then we must have $b_{\ell,s} > m$.  Now if $C_{e,s} \neq \{0,1,\dots,k-1\}$, then we enter Case 1 of the construction and define $f_e(s) \notin C_{e,s}$, so $f_e(s) \notin C_e$.  On the other hand, if $C_{e,s} = \{0,1,\dots,k-1\}$, then since $b_{n,s} = b_n \leq m$ for all $n \in C_e$ and $b_{n,s} > m$ for all $n \notin C_e$, it follows that the $\ell$ chosen in Case 2 of the construction must satisfy $\ell \notin C_e$, so $f_e(s) \notin C_e$.

We have therefore shown that $f_e(s) \notin C_e$ for all $s \geq t$.  Since $D_e$ is infinite, we may fix $b \in D_e$ with $b \geq t$.  We then have have $f_e(b) \notin C_e$, contradicting the definition of $C_e$.
\end{proof}

\begin{corollary}\label{C:diagonalize_TSnk}
For each $n \geq 1$ and $k \geq 2$, there exists a computable instance of $\Seq\TS^n_k$ with no $\emptyset^{(n)}$-computable solution.
\end{corollary}

\begin{proof}
We prove the following stronger claim:  For each $n \geq 1$, $k \geq 2$, and $X \in 2^{\omega}$, there exists an $X$-computable instance of $\Seq\TS^n_k$ with no $X^{(n)}$-computable solution.  We fix $k$ and prove this result by induction on $n$.  The base case of $n = 1$ is given by the relativized version of Theorem~\ref{T:diagonalize_TS1k}.  Suppose that we know the result for a fixed $n \geq 1$.  Let $X \in 2^{\omega}$ be arbitrary.  By induction, we may fix an $X'$-computable instance $\seq{g_i : i \in \omega}$ of $\Seq\TS^n_k$ with no $X^{(n+1)}$-computable solution.  By the relativized Limit Lemma, we may fix an $X$-computable sequence $\seq{f_i : i \in \omega}$ such that $g_i(\tuple{x}) = \lim_s f_i(\tuple{x},s)$ for all $i$ and all $\tuple{x}$.  We may assume that $f_i \colon [\omega]^{n+1} \to k$ for each $i$, and hence that $\seq{f_i : i \in \omega}$ is an $X$-computable instance of $\Seq\TS^{n+1}_k$.  Now if $\seq{T_i : i \in \omega}$ is a solution to $\seq{f_i : i \in \omega}$, then each $T_i$ is an infinite thin set for $f_i$, so each $T_i$ is an infinite thin set for $g_i$, and hence $\seq{T_i : i \in \omega}$ is a solution to $\seq{g_i : i \in \omega}$.  Therefore, $\seq{f_i : i \in \omega}$ has no $X^{(n+1)}$-computable solution.  This completes the induction.
\end{proof}

\noindent
After relativization and translation into the language of Weihrauch reducibility, we obtain the following.

\begin{corollary}\label{C:TJ}
For each $n \geq 1$ and $k \geq 2$, $\Seq\TS^n_k \nred \mathsf{TJ}^n$.
\end{corollary}

\noindent
Here and henceforth, $\mathsf{TJ}^n$ denotes the multi-valued function that sends each set $X$ to its Turing jump $X^{(n)}$.
From the perspective of reverse mathematics, the corresponding $\Pi^1_2$-principles $\mathsf{TJ}^n$ for standard $n \geq 1$, are all equivalent to arithmetic comprehension and hence indistinguishable from each other.

\subsection{Infinitely Many Colors}

Although the principles $\Seq\TS^1_k$ for $k \in \omega$ appear to behave similarly with regards to diagonalizing and coding, the situation for $\Seq\TS^1_{\omega}$ is very different.  We first prove the following result that contrasts with Theorem~\ref{T:diagonalize_TS1k}.

\begin{proposition}
Every computable instance of $\Seq\TS^1_{\omega}$ has a $\emptyset'$-computable solution.
\end{proposition}

\begin{proof}
Let $\seq{f_i : i \in \omega}$ be a computable instance of $\Seq\TS^1_{\omega}$.  Using $\emptyset'$ as an oracle, we compute a sequence of thin sets $\seq{A_i : i \in \omega}$.  We define each $A_i = \{a_{i,m} : m \in \omega\}$ independently by using $\emptyset'$ to compute an increasing sequence $a_{i,0} < a_{i,1} < \dots$.  Given $i$, start by asking $\emptyset'$ if there exists $b \in \omega$ such that $f_i(b) \neq 0$.  Let $a_{i,0}$ be the least such $b$ if one exists, and otherwise let $a_{i,0} = 0$.  Suppose that we have defined $a_{i,n}$.  Ask $\emptyset'$ if there exists $b > a_{i,n}$ such that $f_i(b) \neq 0$.  Let $a_{i,n+1}$ be least such $b$ if one exists, and otherwise let $a_{i,n+1} = a_{i,n} + 1$.  Let
\[
A_i = \{a_{i,n} : n \in \omega\}
\]
and notice that $\seq{A_i : i \in \omega}$ is $\emptyset'$-computable.

Let $i \in \omega$.  We claim that $A_i$ is thin for $f_i$.  If the set $\{b \in \omega : f_i(b) \neq 0\}$ is infinite, then $A_i \subseteq \{b \in \omega : f_i(b) \neq 0\}$, so $A_i$ is thin for $f_i$.  On the other hand, if $\{b \in \omega : f_i(b) \neq 0\}$ is finite, then $\text{range}(f_i)$ is finite, so $A_i$ is trivially thin for $f_i$.  Therefore, $\seq{A_i : i \in \omega}$ is a $\emptyset'$-computable solution to $\seq{f_i : i \in \omega}$.
\end{proof}

\noindent
After relativization and translation into the language of Weihrauch reducibility, we obtain the following.

\begin{corollary}
$\Seq\TS^1_{\omega} \sred \mathsf{TJ}^1$.
\end{corollary}

Finally, we have a strong non-coding result for infinitely many colors to contrast with Corollary~\ref{C:coding_into_seqTSnk}.  Recall that given two degrees ${\bf a}$ and ${\bf b}$, the notation ${\bf a} \gg {\bf b}$ means that every infinite subtree of $2^{<\omega}$ of degree $\mathbf{b}$ has an infinite path of degree at most $\mathbf{a}$. (See~\cite[pp. 10--11]{CJS-2001} for some of the basic properties of this relation.)

\begin{theorem}
Every computable instance of $\Seq\TS^1_{\omega}$ has a low$_2$ solution.  In fact, if ${\bf d} \gg {\bf 0}'$, then every computable instance of $\Seq\TS^1_{\omega}$ has a solution $A$ such that $\text{deg}(A)' \leq {\bf d}$.
\end{theorem}

\begin{proof}
Let $f = \seq{f_i : i \in \omega}$ be a computable instance of $\Seq\TS^1_{\omega}$, and fix $\mathbf{d} \gg \0'$. We obtain the solution $A = \seq{A_i : i \in \omega}$ generically for the following notion of forcing, $\mathbb{P} = (P,\leq)$. An element of $P$ is a pair $\langle \sigma, \tau \rangle$ where $\sigma \in 2^{<\omega}$ and $\tau \in \omega^{<\omega}$, such that:
\begin{itemize}
\item for all $i$ and all $x$, if $\sigma(\seq{i,x})\downarrow = 1$ and $\tau(i)\downarrow$, then $\sigma(\seq{i,x}) \neq \tau(i)$;
\item for all $i$, if $\tau(i)\downarrow$, then the set $\{x \in \omega : f_i(x) \neq \tau(i)\}$ is infinite.
\end{itemize}
We think of $\sigma$ as being broken into columns $\sigma = \seq{\sigma_i : i \in \omega}$ and as being a finite initial segment of the resulting $A = \seq{A_i : i \in \omega}$. Thus, $\sigma(\seq{i,x}) = \sigma_i(x)$, and $\sigma_i$ is a initial segment of $A_i$. The finite sequence $\tau$ says which colors are being omitted on a given column, i.e.,~if $\tau(i)\downarrow$, then the resulting $A_i$ will have the property that $f_i(a) \neq \tau(i)$ for all $a \in A_i$. We define $\langle \sigma^*,\tau^* \rangle \leq \langle \sigma,\tau \rangle$ if both $\sigma \preceq \sigma^*$ and $\tau \preceq \tau^*$.

From a sufficiently generic sequence of conditions
\begin{equation}\label{E:generic}
\seq{\sigma_0,\tau_0} \geq \seq{\sigma_1,\tau_1} \geq \cdots
\end{equation}
we can compute the set $A = \bigcup_{i \in \omega} \sigma_i$ and $B = \bigcup_{i \in \omega} \tau_i$, and $A$ will clearly be a solution to $f$. The appropriate level of genericity corresponds to meeting the following requirements for all $e,i \in \omega$:
\[
\begin{array}{rll}
\mathcal{Q}_{e,i} & : & |A_i| \geq e;\\
\mathcal{R}_{i} & : & B(i) \textrm{ is defined}.
\end{array}
\]
We wish to obtain such a sequence with jump of degree at most $\mathbf{d}$, so we also have the requirement:
\[
\begin{array}{rll}
\mathcal{S}_{e} & : A'(e) \textrm{ is forced}.
\end{array}
\]
It thus suffices to show that these requirements are $\mathbf{d}$-effectively dense, i.e.,~that we can use ${\bf d}$ to extend a given condition $\langle \sigma,\tau \rangle$ to meet a given one of the above requirements.

First, suppose we wish to meet $\mathcal{Q}_{i,e}$. If $\tau(i)\uparrow$, we effectively extend $\sigma$ to $\sigma^*$ by adding $e$ many $1$s in the $i$th column, and only $0$s in the other columns. Then $\seq{\sigma^*,\tau}$ is the desired extension. If $\tau(i) \downarrow$, we can computably find distinct $x_0,\ldots,x_{e-1} > |\sigma_i|$ with $f_i(x_0) \neq \tau(i), \ldots, f_i(x_{e-1}) \neq \tau(i)$. (This is possible because $\seq{\sigma,\tau}$ is a condition, so $\{x \in \omega : f_i(x) \neq \tau(i)\}$ is infinite.) We effectively extend $\sigma$ to $\sigma^*$ so that $\sigma_i^*(x_0) = \cdots = \sigma_i^*(x_{e-1}) = 1$, and all other new bits of $\sigma^*$ are $0$. We then take $\seq{\sigma^*,\tau}$ for the extension. Clearly, in either case, if $A$ extends $\sigma^*$ then $A$ satisfies~$\mathcal{Q}_e$.

Next, suppose we wish to meet $\mathcal{R}_i$. If $\tau(i) \downarrow$, we can just keep $\seq{\sigma,\tau}$, so suppose otherwise. Since the set $F = \{x \in \omega : \sigma_i(x)\downarrow = 1\}$ is finite, so is the set $C = \{c \in \omega : (\exists x \in F)f_i(x) = c\}$, and we can find a canonical index for it effectively from $i$ and $\seq{\sigma,\tau}$. Fix $c_0,c_1 \notin C$ with $c_0 \neq c_1$.  Now at least one of the two sets $\{x \in \omega : f_i(x) \neq c_0\}$ or $\{x \in \omega : f_i(x) \neq c_1\}$ must be infinite. Since these are two effectively given $\Pi_2^0$ sets, Lemma 4.2 of~\cite{CJS-2001} implies that we can ${\bf d}$-effectively determine a $k \in \{0,1\}$ such that $\{x \in \omega : f_i(x) \neq c_k\}$ is infinite.  If we let $\tau^*$ be $\tau$ extended so that $\tau^*(i) = c_k$, then $\seq{\sigma,\tau^*}$ is a condition by choice of $k$, and so we can take it to be our extension. Clearly, if $B$ extends $\tau^*$ then $B(i)$ is defined.

Finally, suppose we wish to meet $\mathcal{S}_e$. Notice that the set of $\sigma^* \succeq \sigma$ that respect $\tau$, i.e.,~the set of $\sigma^* \succeq \sigma$ such that $\langle \sigma^*,\tau \rangle$ is a condition, is computable and we can find an index for it as such effectively from $e$ and $\langle \sigma,\tau \rangle$.  Since ${\bf d} \gg {\bf 0}'$, we have that ${\bf d} \geq {\bf 0}'$, and hence ${\bf d}$ can determine if there exists such a $\sigma^*$ with $\Phi_e^{\sigma^*}(e)\downarrow$.  If so, we extend to the condition $\langle \sigma^*,\tau \rangle$, and otherwise we keep $\langle \sigma,\tau \rangle$.  Notice that so long as $A$ extends $\sigma^*$, then in the former case we will have $e \in A'$, while in the latter we will have $e \notin A'$.

The argument is now put together in the usual way. We let $\seq{\sigma_0,\tau_0} = \seq{\emptyset,\emptyset}$, and then repeatedly use the density of the requirements to $\mathbf{d}$-effectively produce the sequence in \eqref{E:generic}.
\end{proof}

\noindent
After relativization and translation into the language of Weihrauch reducibility, we obtain the following.

\begin{corollary}
$\mathsf{TJ}^1 \after \Seq\TS^1_{\omega} \sred \WKL \after \mathsf{TJ}^1$.
\end{corollary}

\noindent
Using the uniform low basis theorem of Brattka, de Brecht, and Pauly~\cite{BBP-2012}, we also obtain the following.

\begin{corollary}
$\Seq\TS^1_{\omega} \sred \mathfrak{L}_2$.
\end{corollary}


\noindent
Where the low$_2$ principle $\mathfrak{L}_2$ is the composition of the inverse function of $\mathsf{TJ}^2$ with the iterated limit operator $\mathsf{lim}^2$ in a manner similar to the definition of $\mathfrak{L}$ in \cite{BBP-2012}.

By looking for splittings instead of forcing the jump, one can also prove a cone-avoidance theorem saying that if $\seq{C_j : j \in \omega}$ is a sequence of non-computable sets, then every computable instance of $\Seq\TS^1_{\omega}$ has a solution $A$ such that $C_j \nleq_T A$ for all $j$.  As in~\cite[Theorem 3.4]{DJ-2009}, it is also possible to combine these arguments to produce low$_2$ cone-avoiding solutions to computable instances in the case that the sequence $\seq{C_j : j \in \omega}$ is $\emptyset'$-computable. We omit the details, which are standard.

\subsection{Weihrauch Reductions for Thin Sets}

We are not able to prove an analog of Theorem~\ref{T:Ramsey_non-uniform} for $\TS^n_k$ using the Squashing Theorem because we lack an analog of Proposition~\ref{P:combining_ramsey} for thin sets.  However, we can give a direct proof that $\TS^1_j \nred \TS^1_k$ when $j < k$.

\begin{theorem}\label{T:TS1_nonsquashing}
For all $j,k \geq 2$ with $j < k$, we have $\TS^1_j \nred \TS^1_k$.
\end{theorem}

\begin{proof}
Suppose instead that $\TS^1_j \red \TS^1_k$, and fix reductions $\Phi$ and $\Psi$ witnessing this fact. We shall build a computable coloring $f : \omega \to j$ and an infinite thin set $T \subseteq \omega$ for $\Phi(f) : \omega \to k$ such that $\Psi(T)$ cannot be an infinite thin set for $f$. For convenience, we assume that for any finite $F \subseteq \omega$, if $\Psi(F)(x) \downarrow$ for some $x \in \omega$ then the use of the computation is bounded by $\max F$. As the construction will be uniformly computable, we may use the recursion theorem as in the proof of Proposition \ref{P:qWWKL} so as to not have to consider the instance $f$ in the oracle for $\Psi$. 

The idea of the proof is to start by defining $f$ arbitrarily, monitoring the coloring $\Phi(f)$ as it forms alongside, and waiting to find a finite homogeneous set $F_0$ for $\Phi(f)$ that is large enough so that $\Psi(F_0)$ contains some number $x_0$. Once this happens, we change how we define $f$ so that all future numbers have a different color from $x_0$. In this way, we force any sufficiently large set extending $\Psi(F_0)$ to contain numbers of at least two different colors. We then repeat the process, looking for a finite homogeneous set $F_1 > F_0$ for $\Phi(f)$ large enough so that $\Psi(F_0 \cup F_1)$ contains some number $x_1$ colored differently from $x_0$. We then change how we define $f$ again so that all future numbers have a different color from $x_0$ and from $x_1$. In this way, we force any sufficiently large set extending $\Psi(F_0 \cup F_1)$ to contain numbers of at least three different colors.

Continuing in this way, we build $F_0 < \cdots < F_{j-2}$ such that any sufficiently large set extending $\Psi(\bigcup_{i < j-1} F_i)$ contains numbers of all $j$ many colors. Thus, to define the desired set $T$, we have only to produce an infinite set extending $\bigcup_{i < j-1} F_i$ that is thin for $\Phi(f)$. But since each $F_i$ was chosen to be homogeneous for $\Phi(f)$, this coloring assumes at most $j-1$ many colors on $\bigcup_{i < j-1} F_i$. So, if we let $H > F_{j-2}$ be any infinite homogeneous set for $\Phi(f)$, then $\Phi(f)$ assumes at most $j$ many colors on $\bigcup_{i < j-1} F_i \cup H$, which we take to be $T$. Then $T$ is thin for $\Phi(f)$ since $j < k$.

We proceed to the formal details.

\medskip
\noindent \emph{Construction.} We proceed by stages. At stage $s$, we define an initial segment $f_s$ of $f$ on $[0,s]$. During the construction, we also define $j-1$ many sets $F_0,\ldots,F_{j-2}$ that will be used in the definition of $T$.

At stage $s = 0$ we set $f_0 = \emptyset$, and declare all colors $c < j$ \emph{valid}.

At stage $s > 0$, let $l \in \omega$ be such that we have already defined $F_i$ for each $i < l$. Call $s$ an \emph{action stage} if $l < j-1$, and if there exists a finite set $F \leq s$ and number $x \leq s$ such that
\begin{itemize}
\item $\bigcup_{i < l} F_i < F$;
\item $F$ is homogeneous for $\Phi(f_{s-1})$;
\item $f_{s-1}(x)$ is some currently valid color;
\item $\Psi(\bigcup_{i < l} F_i)(x) \uparrow$ and $\Psi(\bigcup_{i < l} F_i \cup F)(x) \downarrow = 1$.
\end{itemize}
In this case, let $F_l$ and $x_l$ be the least such $F$ and $x$, respectively, and declare the color $f_{s-1}(x)$ to no longer be valid. By induction, this leaves at least one valid color.

Regardless of whether $s$ is an action stage or not, we extend $f_{s-1}$ to $f_s$ by choosing the least color $c < j$ that is still valid, and letting $f_s(y) = c$ for all $y \leq s$ on which $f_{s-1}$ has not yet been defined.

\medskip
\noindent \emph{Verification.} It is clear that $f = \bigcup_s f_s$ is a computable $j$-coloring. We begin with an observation. Note that there can be no more than $j-1$ many action stages, since the number of $F_i$ defined at the start of such a stage must be fewer than $j-1$, and a new $F_i$ is then defined. So, let $l \leq j-1$ be the total number of action stages; we claim that $l = j-1$. Since the number of valid colors is reduced by one at every action stage, this implies that there is precisely one color that is permanently valid.

Before proving the claim, we define $T$. Since each $F_i$ for $i < l$ is homogeneous for $\Phi(f)$, it follows that $\Phi(f)$ assumes at most $l$ many colors on $\bigcup_{i < l} F_i$. Thus if $H > F_{l-1}$ is any infinite homogeneous set for $\Phi(f)$, then $\Phi(f)$ assumes at most $l+1 \leq j < k$ many colors on $\bigcup_{i < l} F_i \cup H$. It follows that $T = \bigcup_{i < l} F_i \cup H$ is thin for $\Phi(f)$.

Now to see the claim, let $t$ be $0$, or any action stage before the $(j-1)$st. Seeking a contradiction, suppose there is no action stage greater than $t$. In particular, all the $F_i$ for $i < l$ are defined at or before stage $t$. If $c < j$ is the least color still valid at the end of stage $t$, then all sufficiently large numbers are colored $c$ by $f$. Thus, since $T$ is an infinite thin set for $\Phi(f)$, it follows that $\Psi(T)$, being an infinite thin set for $f$, contains some number $x$ colored $c$ by $f$ on which $\Psi(\bigcup_{i < l} F_i)$ diverges. But now if $F$ is a sufficiently long initial segment of $H$ so that $\Psi(\bigcup_{i < l} F_i \cup F)(x) \downarrow = 1$, then any stage $s > t$ with $s \geq F$ and $s \geq x$ will be an action stage. The proof is complete.
\end{proof}

Since the circulation of a pre-print of the present article, Hirschfeldt and Jockusch have extended the above argument to all $n > 1$. Their proof will appear in~\cite{HJ-TA}.

\begin{theorem}[Hirschfeldt and Jockusch~\cite{HJ-TA}]
For all $n > 1$ and all $j,k \geq 2$ with $j < k$, we have $\TS^n_j \nred \TS^n_k$.
\end{theorem}

\section{The Rainbow Ramsey's Theorem and measure}\label{S:RRT}

For our final results, we turn to the rainbow Ramsey's theorem and further connections with randomness.

\begin{definition}
Fix $n,k \geq 1$.
\begin{enumerate}
\item A coloring $f \colon [\omega]^n \to \omega$ is \emph{$k$-bounded} if for each $c \in \omega$, there are at most $k$ many $\tuple{x} \in [\omega]^n$ such that $f(\tuple{x}) = c$.
\item A set $S \subseteq \omega$ is a \emph{rainbow for $f$} if $f$ is injective on $[S]^n$.
\end{enumerate}
\end{definition}

\begin{statement}[Rainbow Ramsey's Theorem] Given $n,k \geq 1$, let $\RRT^n_k$ denote the statement every $k$-bounded $f \colon [\omega]^n \to \omega$ has an infinite rainbow. Let $\RRT^n_{<\infty}$ denote $(\forall k \geq 1)~\RRT^n_k$.
\end{statement}

Just as for $\RT^1_k$ and $\TS^1_k$, every computable instance of $\RRT^1_k$ has a computable solution.  However, in contrast to the situations for $\Seq\RT^1_k$ and $\Seq\TS^1_k$, every computable instance of $\Seq\RRT^1_k$ also has a computable solution.  In fact, we have the following stronger fact.

\begin{proposition}
Every computable instance of $\Seq\RRT^1_{<\infty}$ has a computable solution.
\end{proposition}

\begin{proof}
Let $\seq{f_i : i \in \omega}$ be a computable instance of $\Seq\RRT^1_{<\infty}$.  We then have that for each $i$ and each $c$, the set $\{x \in \omega : f_i(x) = c\}$ is finite.  From this it follows that for each $i$ and each \emph{finite} set $C \subseteq \omega$, the set $\{x \in \omega : f_i(x) \in C\}$ is finite.  We can now define a computable sequence $\seq{A_i : i \in \omega}$ by choosing the elements of each $A_i$ recursively so that the color of a new element is distinct from all previous elements already chosen to be $A_i$.
\end{proof}

\begin{theorem}[Csima and Mileti~\cite{CM-2009}, Theorem 3.10]\label{T:Csima_Mileti}
For all $k \geq 1$, if $X \subseteq \omega$ is $2$-random then every computable $k$-bounded coloring $f \colon [\omega]^2 \to \omega$ has an infinite $X$-computable rainbow.
\end{theorem}

The proof of this theorem proceeds by constructing a $\emptyset'$-computable subtree $T$ of $2^{<\omega}$ of positive measure, each infinite path through which computes an infinite rainbow for $f$. This proof is very nearly uniform®. The tree $T$ can be obtained uniformly $\emptyset'$-computably from an index for $f$, and the reduction from the infinite paths through $T$ to the infinite rainbows for $f$ is uniform as well. The only non-uniformity stems from the way $2$-random sets pick out infinite paths through $T$. We begin by showing that this non-uniformity is essential.

For each $i \in \omega$ and each bounded coloring $f \colon [\omega]^2 \to \omega$, let
\[
\mathcal{S}_{f,i} = \{ S \subseteq \omega : \Phi_i(S) \textrm{ is an infinite rainbow for } f\}.
\]
Let $\mu$ denote the uniform measure on Cantor space.

\begin{proposition}\label{P:Csima_Mileti_non-uniform}
There is no computable function $h$ such that for all $i \in \omega$ and all $2$-random $R \subseteq \omega$, if $\Phi_i$ is a $2$-bounded coloring $[\omega]^2 \to \omega$ then $\Phi_{h(i)}(R)$ is an infinite rainbow for $f$.
\end{proposition}

\begin{proof}
First, fix any $i \in \omega$. Let $w(i)$ be the least $\sigma \in 2^{<\omega}$, if one exists, such that
\begin{equation}\label{E:convergence}
\Phi_i(\sigma)(x) \downarrow = \Phi_i(\sigma)(y) \downarrow = 1
\end{equation}
for some $x < y$. Then, define a coloring $f_i \colon [\omega]^2 \to \omega$ by stages, as follows. At stage $s$, we define $f_i$ on $[0,s) \times \{s\}$. If $w(i)$ has not yet converged, let $f_i(z,s) = \seq{z,s}$ for all $z < s$. Otherwise, choose the least $x < y < s$ satisfying \eqref{E:convergence} above for $\sigma = w(i)$, and define
\[
f_i(z,s) =
\begin{cases}
\seq{x,s} & \text{if } z = x \text{ or } z = y,\\
\seq{z,s} & \text{else,}
\end{cases}
\]
for all $z < s$.

Clearly, $f_i$ is $2$-bounded for each $i$. Moreover, if there exists an $S \subseteq \omega$ such that $\Phi_i(S)$ is an infinite rainbow for $f_i$, then $w(i)$ is defined. Say $w(i) = \sigma$. Then for the least $x < y$ satisfying \eqref{E:convergence}, we have $f_i(x,s) = f_i(y,s)$ for all sufficiently large $s$, so $x$ and $y$ cannot belong to any infinite rainbow for $f_i$. In particular, if $S \succeq \sigma$ then $\Phi_i(S)$ is not such a rainbow for $f_i$. It follows that
\begin{equation}\label{E:measure}
\mu(\mathcal{S}_{f_i,i}) \leq 1 - 2^{-|\sigma|} < 1.
\end{equation}

Now note that $f_i$ is uniformly computable in $i$. So let $g$ be a computable function such that $f_i = \Phi_{g(i)}$ for all $i$. Seeking a contradiction, suppose a function $h$ as in the statement exists. By the recursion theorem, we may fix an $i \in \omega$ such that $\Phi_{h(g(i))}(S) = \Phi_i(S)$ for all $S \subseteq \omega$. In particular, $\mathcal{S}_{f_i,h(g(i))} = \mathcal{S}_{f_i,i}$, so by assumption, $\mathcal{S}_{f_i,i}$ contains all $2$-random subsets of $\omega$. But since the set of all $2$-random subsets of $\omega$ has measure $1$, this contradicts \eqref{E:measure}.
\end{proof}

We wish to know whether Theorem~\ref{T:Csima_Mileti} carries over to $\omega$ applications, i.e., whether every computable instance of $\Seq\RRT^2_k$ also has a solution computable in each $2$-random. By the preceding proposition, the most direct way of obtaining this fails, as the theorem cannot be proved uniformly. Nevertheless, we are able to give an affirmative answer to the question.

\begin{theorem}\label{T:SeqRRT_2-randoms}
If $k \geq 1$ and $X \subseteq \omega$ is $2$-random, then every computable instance of $\Seq\RRT^2_k$ has an $X$-computable solution.
\end{theorem}

\begin{proof}
This is a small adaptation of the proof of Theorem 3.10 in~\cite{CM-2009}, and we refer to results in that article.  Let $f = \seq{f_i : i \in \omega}$ be a computable instance of $\Seq\RRT^2_k$, so $f_i \colon [\omega]^2 \to \omega$ is $k$-bounded for all $i$.  The proof of Proposition 3.3 is uniform, so we may assume that each $f_i$ is normal.  When defining $\varphi_f$ and $T_f$ in Definition 3.7 and Definition 3.8, instead interleave the process of working on the various $f_i$ across the levels of the tree, i.e.,~at level $\langle i,n \rangle$, work on the function $f_i$.  Proposition 3.9 still applies so that any $2$-random $X$ will compute a path through this combined tree, and any such path computes a solution to $\seq{f_i : i \in \omega}$.
\end{proof}

To translate the above result into the language of Weihrauch reducibility or the language of reverse mathematics, we first need to isolate the assertion of the existence of 2-randoms as a $\Pi^1_2$ principle. To do so in a formal setting takes some care, since this is intrinsically a statement about paths through non-computable trees. A detailed account of this and associated difficulties in the specific context of second-order arithmetic is presented in Avigad, Dean, and Rute \cite[Section 3]{ADR-2012}, where they also introduce the principle $2\mhyphen\WWKL$ as one possible formalization (not to be confused with the principle $q\mhyphen\WWKL$ discussed in Section~\ref{S:WWKL}). Here we shall content ourselves with the informal definition below (which agrees with theirs in $\omega$-models) and refer the reader to their paper for technical details.

\begin{statement}[$2\mhyphen\WWKL$]
For every set $X$, there is a $2$-random set relative to $X$.
\end{statement}

\noindent 

Thus, after relativization and translation into the language of Weihrauch reducibility, we obtain from Theorem \ref{T:SeqRRT_2-randoms} the following:

\begin{corollary}
For each $k \geq 1,$ $\Seq\RRT^2_k \red 2\mhyphen\WWKL$ and $\RCA \vdash 2\mhyphen\WWKL \rightarrow \Seq\RRT^2_k$.
\end{corollary}

The following consequence complements a result of Conidis and Slaman \cite[Theorem 2.1 and Corollary 4.2]{CS-2013} that $\RRT^2_k$ does not imply $\mathsf{B}\Sigma^0_2$ over $\RCA$.

\begin{corollary}
For each $k \geq 1$, $\RT^1_2 \nred \RRT^2_k$.
\end{corollary}

\begin{proof}
Suppose instead that $\RT^1_2 \red \RRT^2_k$.  By Proposition~\ref{P:unif_imlies_seq_unif}, this would imply that $\Seq\RT^1_2 \red \Seq\RRT^2_k$.  By Lemma~\ref{L:jump_coding}, there is a computable instance of $\Seq\RT^1_2$ such that every solution computes $\emptyset'$.  By Theorem~\ref{T:SeqRRT_2-randoms}, every $2$-random $X \subseteq \omega$ computes a solution to every computable instance of $\Seq\RRT^2_k$.  This is a contradiction because there is $2$-random that does not compute $\emptyset'$ (in fact, no $2$-random computes $\emptyset'$).
\end{proof}

For our final result, we exhibit a degree-theoretic difference between $\Seq\RRT^2_k$ and $\Seq\RRT^2_{<\infty}$. This contrasts with the situation between $\Seq\RT^2_k$ and $\Seq\RT^2_{<\infty}$, i.e.,~the sequential forms of Ramsey's Theorem for $k$ many colors and finitely many colors. Specifically, it is not difficult to see that if $X$ is a set with $\deg(X) \gg \0''$ then every computable instance of either of these principles has an $X$-computable solution. That this bound is sharp follows by recent work of Wang~\cite[Section 3.1]{Wang-TA}.

\begin{lemma}\label{L:arb_bounds}
For each rational number $q > 0$ and each $i \in \omega$, there exists a bounded coloring $f \colon [\omega]^2 \to \omega$ such that $\mu(\mathcal{S}_{f,i}) < q$. Moreover, an index for $f$ as a computable function can be found uniformly computably from $q$ and $i$.
\end{lemma}

\begin{proof}
The idea is to elaborate on the proof of Proposition~\ref{P:Csima_Mileti_non-uniform}. For all $i,n \in \omega$, we inductively define $w(i,n)$ to be the least canonical index of a finite subset $F$ of $2^{<\omega}$ such that
\begin{enumerate}
\item $\cyl{F} \cap \bigcup_{m < n} \cyl{D_{w(i,m)}} = \emptyset$;
\item $\mu(\cyl{F}) \geq q$;
\item for each $\sigma \in F$, there exist $x < y$ such that $\Phi_i(\sigma)(x) \downarrow = \Phi_i(\sigma)(y) \downarrow = 1$, and $x$ and $y$ are not used by $D_{w(i,m)}$ for any $m < n$, as defined below.
\end{enumerate}
For each $\sigma$ in $F$, choose the least $x$ and $y$ satisfying condition 3, and say these are \emph{used} by $\sigma$ and by $F$.

The coloring $f$ is now defined by stages. At stage $s$, we define $f$ on $[0,s) \times \{s\}$. Choose the least $n$ such that $w(i,n)$ has not yet converged. For each $m < n$, and each $\sigma \in D_{w(i,m)}$, choose the $x < y$ used by $\sigma$, and define $f(x,s) = f(y,s) = \seq{x_0,s}$ for the least $x_0$ used by $D_{w(i,m)}$. (We may assume that if $w(i,m)$ has converged by stage $s$ then all numbers used by $D_{w(i,m)}$ are smaller than $s$.) For $z < s$ not used by any $D_{w(i,m)}$, let $f(z,s) = \seq{z,s}$.

Clearly, $f$ is computable. We claim that it is bounded. To this end, observe that $w(i,n)$ is defined for  only finitely many $n$, since otherwise
\[
D_{w(e,0)},\ldots,D_{w(e,\lceil 1/q \rceil)}
\]
would determine $\lceil 1/q \rceil+1$ many disjoint subsets $2^{\omega}$, each of measure at least $q$. So let $n$ be least such that $w(i,n)$ is undefined, and for each $m < n$, let $k_m$ be the number of elements used by $D_{w(i,m)}$. The only colors used more than once by $f$ are of the form $\seq{x_0,s}$, where $x_0$ is the least number used by some $D_{w(i,m)}$, and in this case, $f(x,t) = \seq{x_0,s}$ only if $s = t$ and $x$ is used by $D_{w(i,m)}$. Thus, $f$ uses each such color $\seq{x_0,s}$ at most $k_m$ many times, implying that $f$ is $k$-bounded for $k = \sup_{m < n} k_m$.

Now with $n$ as above, notice that if an $S \subseteq \omega$ extends some $\sigma \in D_{w(i,m)}$ for $m < n$, then $\Phi_i(S)(x) \downarrow = \Phi_i(S)(y) \downarrow = 1$ for the $x < y$ used by $\sigma$. By construction, $f(x,s) = f(y,s)$ for all sufficiently large $s$, so $\Phi_i(S)$ cannot be an infinite rainbow for $f$. Thus, any $S$ such that $\Phi_i(S)$ is such a rainbow must lie outside of $\bigcup_{m < n} \cyl{D_{w(i,m)}}$. But this means that the measure of all such $S$ is less than $q$, because otherwise we could find a finite set $F$ satisfying conditions 1, 2, and 3 in the definition of $w$, and $w(i,n)$ would be defined.
\end{proof}

\begin{proposition}
There exists a computable instance of $\Seq\RRT^2_{<\infty}$ such that not every $2$-random $X \subseteq \omega$ computes a solution.
\end{proposition}

\begin{proof}
Let $g$ be a computable function such that
\[
\Phi_{g(e,j)}(S)(x) = \Phi_e(S)(\seq{x,\seq{e,j}})
\]
for all $e,j \in \omega$ and all $S \subseteq \omega$. In other words, $\Phi_{g(e,j)}(S)$ is the restriction of $\Phi_e(S)$ to the $\seq{e,j}$th column. For all $e,j \in \omega$, apply Lemma~\ref{L:arb_bounds} to get a computable bounded coloring $f_{\seq{e,j}} \colon [\omega]^2 \to \omega$ such that
\[
\mu(\mathcal{S}_{f_{\seq{e,j}},g(e,j)}) < 2^{-j}.
\]
Then $\seq{f_i : i \in \omega}$ is a computable sequence of colorings, and further, for all $e \in \omega$ and $S \subseteq \omega$, if $\Phi_e(S)$ is a sequence of infinite rainbows for the $f_i$, then $\Phi_{g(e,j)}(S)$ is an infinite rainbow for $f_{\seq{e,j}}$. Thus for each $e$, it must be that
\[
\mu(\{S \subseteq \omega : \Phi_e(S) \textrm{ is a sequence of infinite rainbows for the } f_i\}) = 0,
\]
for if this measure were at least $2^{-j}$ then so would $\mu(\mathcal{S}_{f_{\seq{e,j}},g(e,j)})$, which cannot be. Since the measure of the $2$-randoms is $1$, it follows that there is a $2$-random $X \subseteq \omega$ that computes no sequence of infinite rainbows for the $f_i$.
\end{proof}

\noindent
After relativization and translation into the language of Weihrauch reducibility, we obtain the following.

\begin{corollary}
$\Seq\RRT^2_{<\infty} \nred 2\mhyphen\WWKL$.
\end{corollary}

\section{Questions}

We close by listing a few questions left open by our work. Chief among these is whether the analogue of Theorem \ref{T:Ramsey_non-uniform} holds for (general) Weihrauch reducibility.

\begin{question}
If $n,j,k \geq 2$ and $j < k$, is it the case that $\RT^n_k \nred \RT^n_j$?
\end{question}

Though not our focus here, our results naturally lead to questions about non-uniform reductions as well. In particular, we can ask the following about a non-uniform version of Theorem~\ref{T:Ramsey_non-uniform}, which is closely related to Question 5.5.3 of~\cite{Mileti-2004}.

\begin{question}
If $n,j,k \geq 2$ and $j < k$, does every $f : [\omega]^n \to k$ compute a $g : [\omega]^n \to j$, such that every infinite homogeneous set for $g$ computes an infinite homogeneous set for $f$? 
\end{question}

We also have the following question about thin sets and rainbows.

\begin{question}
Are there analogues of Proposition~\ref{P:combining_ramsey} for $\TS^n_k$ and $\RRT^n_k$?
\end{question}

\appendix
\section{Equivalence of definitions}\label{A:equivalence}

In this section, we provide a proof of the equivalence of Definition \ref{D:uniform_reductions} with the definition of (strong) Weihrauch reducibility employed in computable analysis, in the limited context of where both reductions make sense. We do not include here the technicalities particular to that field, and instead focus on the following primary definition that is used to extend the notion of Weihrauch reducibility to more specific settings.

Our discussion below will be limited to functions from Cantor space to Cantor space, but nothing would be lost by considering instead functions on Baire space.

\begin{definition}\label{D:CA_Weihrauch_defn}
Let $\mathcal{F}$ and $\mathcal{G}$ be sets of partial functions $2^{\omega} \to 2^{\omega}$.
\begin{enumerate}
\item $\mathcal{F}$ is \emph{Weihrauch reducible} to $\mathcal{G}$, written $\mathcal{F} \red \mathcal{G}$, if there exist Turing functionals $\Phi$ and $\Psi$ such that
\[
(\forall G \in \mathcal{G})(\exists F \in \mathcal{F})~F = \Psi \circ \seq{\mathrm{id}, G \circ \Phi}.
\]
\item $\mathcal{F}$ is \emph{strongly Weihrauch reducible} to $\mathcal{G}$, written $\mathcal{F} \sred \mathcal{G}$, if there exist Turing functionals $\Phi$ and $\Psi$ such that
\[
(\forall G \in \mathcal{G})(\exists F \in \mathcal{F})~F = \Psi \circ G \circ \Phi.
\]
\end{enumerate}
\end{definition}

The order of quantifiers here may at first appear to be reversed from that used in our definition of $\red$ and $\sred$. In order to explain this, we shall use the following notation.
\begin{enumerate}
\item Given a $\Pi^1_2$ principle $\mathsf{P}$ of second-order arithmetic, let $\mathcal{F}_{\mathsf{P}}$ be the set of all partial $F : 2^{\omega} \to 2^{\omega}$ whose domain includes the set of instances of $\mathsf{P}$, and $F(A)$ for each instance $A$ of $\mathsf{P}$ is a solution to the that instance.
\item Given an arithmetically-definable set $\mathcal{F}$ of partial functions $2^{\omega} \to 2^{\omega}$, let $\mathsf{P}_{\mathcal{F}}$ be the $\Pi^1_2$ principle of second-order arithmetic whose instances are the members of the intersection of the domains of the functions in $\mathcal{F}$, and the solutions to any such instance $A$ are the sets $F(A)$ for $F \in \mathcal{F}$.
\end{enumerate}
We begin with the following general Galois connection.

\begin{proposition}
Let $\mathsf{P}$ be a $\Pi^1_2$ principle of second-order arithmetic, and let $\mathcal{F}$ be an arithmetically-definable set of partial functions $2^{\omega} \to 2^{\omega}$ with common domain. Then:
\begin{enumerate}
\item $\mathsf{P} \red \mathsf{P}_{\mathcal{F}}$ if and only if $\mathcal{F}_{\mathsf{P}} \red \mathcal{F}$;
\item $\mathsf{P} \sred \mathsf{P}_{\mathcal{F}}$ if and only if $\mathcal{F}_{\mathsf{P}} \sred \mathcal{F}$.
\end{enumerate}
\end{proposition}

\begin{proof}
We prove~(1), the proof of~(2) being analogous. First, suppose $\mathsf{P} \red \mathsf{P}_{\mathcal{F}}$ via $\Phi$ and $\Psi$. We claim that $\Phi$ and $\Psi$ also witness that $\mathcal{F}_{\mathsf{P}} \red \mathcal{F}$. Indeed, fix any $G \in \mathcal{F}$, and define $F = \Psi \circ \seq{\mathrm{id},G \circ \Phi}$. We have only to verify that $F \in \mathcal{F}_{\mathsf{P}}$. If $A$ is an instance of $\mathsf{P}$, then $\Phi(A)$ is an instance of $\mathsf{P}_{\mathcal{F}}$, meaning an element of the intersection of the domains of the members of $\mathcal{P}$, and so in particular, a member of the domain of $G$. Thus, $G(\Phi(A))$ is defined, and by definition of $\mathsf{P}_{\mathcal{F}}$, this is a solution to the instance $\Phi(A)$, meaning $\Psi(A,G(\Phi(A)))$ is a solution to $A$. Thus, $A$ is in the domain of $F$, and $F(A) = \Psi(A,G(\Phi(A)))$ is a solution to $A$, as needed.

In the other direction, suppose $\mathcal{F}_{\mathsf{P}} \red \mathcal{F}$ via $\Phi$ and $\Psi$. Let $A$ be any instance of $\mathsf{P}$, and so a member of the intersection of the domains of the functions in $\mathcal{F}_{\mathsf{P}}$. By the definition of $\red$ above, $\Phi(A)$ is an element of the domain of every function in $\mathcal{F}$, and so an instance of $\mathsf{P}_{\mathcal{F}}$. Let $S$ be any solution to this instance, so that $S = G(\Phi(A))$ for some $G \in \mathcal{F}$, and then let $F \in \mathcal{F}_{\mathsf{P}}$ be as given for $G$ by the definition of $\red$.  We have that $F(A)$ is a solution to $A$, and $F(A) = \Psi(A,G(\Phi(A))) = \Psi(A,S)$. This completes the proof.
\end{proof}

The proposition allows us to translate results employing Definition \ref{D:uniform_reductions} into results employing Definition \ref{D:CA_Weihrauch_defn}.

\begin{corollary}
Let $\mathsf{P}$ and $\mathsf{Q}$ be $\Pi^1_2$ principles of second-order arithmetic. Then:
\begin{enumerate}
\item $\mathsf{P} \red \mathsf{Q}$ if and only if $\mathcal{F}_{\mathsf{P}} \red \mathcal{F}_{\mathsf{Q}}$;
\item $\mathsf{P} \sred \mathsf{Q}$ if and only if $\mathcal{F}_{\mathsf{P}} \sred \mathcal{F}_{\mathsf{Q}}$.
\end{enumerate}
\end{corollary}

\begin{proof}
By the proposition, we have that $\mathsf{P} \red \mathsf{P}_{\mathcal{F}_{\mathsf{Q}}}$ if and only if $\mathcal{F}_{\mathsf{P}} \red \mathcal{F}_{\mathsf{Q}}$. But it is easily checked that $\mathsf{P}_{\mathcal{F}_{\mathsf{Q}}} = \mathsf{Q}$. This gives us~(1), and the proof of~(2) is analogous.
\end{proof}

Translations in the reverse direction require an additional assumption.

\begin{corollary}
Let $\mathcal{F}$ and $\mathcal{G}$ each be an arithmetically-definable set of partial functions $2^{\omega} \to 2^{\omega}$ with common domain. If $\mathcal{F}_{\mathsf{P}_{\mathcal{F}}} = \mathcal{F}$ then:
\begin{enumerate}
\item $\mathcal{F} \red \mathcal{G}$ if and only if $\mathsf{P}_{\mathcal{F}} \red \mathsf{P}_{\mathcal{G}}$;
\item $\mathcal{F} \sred \mathcal{G}$ if and only if $\mathsf{P}_{\mathcal{F}} \sred \mathsf{P}_{\mathcal{G}}$.
\end{enumerate}
\end{corollary}

\begin{proof}
Again, we only prove~(1). By the proposition, we have that $\mathcal{F} = \mathcal{F}_{\mathsf{P}_{\mathcal{F}}} \red \mathcal{G}$ if and only if $\mathsf{P}_{\mathcal{F}} \red \mathsf{P}_{\mathcal{G}}$, as desired.
\end{proof}

The additional assumption above is natural, as it encompasses most results from computable analysis. Indeed, the primary objects of study in that context are not reductions between arbitrary sets of functions, but rather, reductions between sets of realizers of multi-functions. For completeness, we include definitions of these concepts (see \cite[Definition 2.2]{Brattka-Gherardi-2011wd} for a more technical version better suited for work with represented spaces).

\begin{definition}
Let $f : 2^{\omega} \to 2^{\omega}$ be a partial multi-valued function. A partial single-valued function $F : 2^{\omega} \to 2^{\omega}$ is a \emph{realizer} for $f$, written $F \vdash f$, if $F(A) \in f(A)$ for each $A$ in the domain of $f$. 
\end{definition}

\noindent Notice now that if $f : 2^{\omega} \to 2^{\omega}$ is a partial multi-valued function, then the set $\mathcal{F} = \{ F : F \vdash f\}$ satisfies $\mathcal{F}_{\mathsf{P}_{\mathcal{F}}} = \mathcal{F}$.

\bibliographystyle{plain}
\bibliography{Papers}

\end{document}